\documentclass[10pt,leqno]{amsart}
\usepackage[skip=5pt plus1pt, indent=20pt]{parskip}

\usepackage{palatino}
\usepackage{mathpazo}
\usepackage{amsmath,amssymb,amsxtra,color,calligra,mathrsfs,tcolorbox}

\usepackage{xcolor} 
\colorlet{mdtRed}{red!50!black}
\colorlet{dblue}{blue!50!black}
\usepackage[colorlinks,pagebackref=true]{hyperref}
\hypersetup{linkcolor=dblue,citecolor=dblue,filecolor=dullmagenta,urlcolor=mdtRed}
\renewcommand*{\backref}[1]{}
\renewcommand*{\backrefalt}[4]{[{%
		\ifcase #1 Not cited.%
		\or $\uparrow$~#2.%
		\else $\uparrow$~#2.%
		\fi%
	}]}
\usepackage[all]{xy}
\usepackage{tikz,tikz-cd,tkz-graph,enumerate}
\usetikzlibrary{matrix,arrows,decorations.pathmorphing}

\usepackage[us,12hr]{datetime}
\usepackage{comment}

\DeclareMathOperator{\Spec}{{\rm Spec}}
\DeclareMathOperator{\Id}{{\rm Id}}
\DeclareMathOperator{\Hom}{{\rm Hom}}

\DeclareMathOperator{\Vect}{{\bf Vect}}
\DeclareMathOperator{\ParVect}{{\bf ParVect}}
\DeclareMathOperator{\RPVect}{{\bf RealParVect}}
\DeclareMathOperator{\Con}{{\rm Con}}

\newcommand{\mf}[1]{\mathfrak{#1}}
\newcommand{\mc}[1]{\mathcal{#1}}
\newcommand{\ms}[1]{\mathscr{#1}}
\newcommand{\bb}[1]{\mathbb{#1}}

\numberwithin{equation}{subsection}

\newtheorem{theorem}[equation]{Theorem}

\newtheorem{lemma}[equation]{Lemma}
\newtheorem{proposition}[equation]{Proposition}

\theoremstyle{definition}
\newtheorem{definition}[equation]{Definition}
\newtheorem{remark}[equation]{Remark}

\theoremstyle{theorem}

\makeatletter
\newcommand\fnsymb[1]{\textsuperscript{\@fnsymbol{#1}}}
\newcommand\fnletter[1]{\lowercase{\textsuperscript{\@alph{#1}}}}
\newcommand\fnnum[1]{\textsuperscript{#1}}
\makeatother

\usepackage{marvosym} 
\renewcommand{\email}[2][1]{\thanks{\textit{Email address}#1: \href{mailto:#2}{#2}}}

\renewcommand{\address}[2][1]{\thanks{\textit{Address}#1: #2}} 

\begin{document}

\baselineskip=15.5pt 

\title[Real Structures on Root Stacks and Parabolic Connections]{Real Structures on Root Stacks and Parabolic Connections}

\author[S. Chakraborty]{Sujoy Chakraborty\fnnum{1}}

\address[\fnnum{1}]{Department of Mathematics, 
Indian Institute of Science Education and Research Tirupati, Andhra Pradesh - 517507, India.}

\email[\fnnum{1}]{sujoy.cmi@gmail.com}

\author[A. Paul]{Arjun Paul\fnnum{2}}

\address[\fnnum{2}]{Department of Mathematics and Statistics, 
Indian Institute of Science Education and Research Kolkata, 
Mohanpur - 741 246, Nadia, West Bengal, India.} 

\email[\fnnum{2}]{arjun.paul@iiserkol.ac.in}

\thanks{Corresponding author: Sujoy Chakraborty}



\begin{abstract}
	Let $D$ be a reduced effective strict normal crossing divisor on a smooth complex variety $X$, 
	and let $\mf X$ be the associated root stack over $\mathbb C$. 
	Suppose that $X$ admits an anti-holomorphic involution (real structure) that keeps $D$ invariant. 
	We show that the root stack $\mf X$ naturally admits a real structure compatible with $X$. 
	We also establish an equivalence of categories between the category of real logarithmic connections 
	on this root stack and the category of real parabolic connections on $X$. 
\end{abstract}

\baselineskip=15.5pt 

\date{Last updated on \today\,at \currenttime\,(IST)}

\maketitle 

\tableofcontents

\section{Introduction}
Parabolic vector bundles were introduced by Mehta and Seshadri \cite{MS80} 
on a compact Riemann surface $X$ together with a set $D$ of marked points. 
In this initial formulation, a parabolic bundle is a vector bundle together 
with the data of a flag for each of its fibers over the points of $D$, 
together with a set of real numbers in $[0,1)$, called weights. 
The celebrated Mehta-Seshadri theorem shows a one-to-one correspondence 
between polystable parabolic vector bundles on $X$ of parabolic degree zero 
and the unitary representations of the fundamental group $\pi_1(X\setminus D)$ 
of the punctured Riemann surface. Later, Simpson in \cite{S90} reformulated 
and extended the definition of parabolic bundles to the case of schemes $X$ 
with a normal crossings divisor $D$. When the weights are all rational, and 
therefore can be assumed to be in $\frac{1}{r}\bb Z$ for some integer $r$, 
this amounts to a locally free sheaf $\mc E$ of finite rank, together with 
a filtration
\begin{align*}
	\mc E = \mc E_0\supset \mc E_{\frac{1}{r}}\supset \mc E_{\frac{2}{r}} 
	\supset\cdots\supset \mc E_{\frac{r-1}{r}}\supset \mc E_1 = \mc E(-D)
\end{align*}
Since their introduction, parabolic bundles, parabolic connections and their 
moduli spaces have been an active area of research, and have been used in a 
number of important applications. Even though most of the progress has been 
in the case of curves, more recently, parabolic connections have made appearance 
in the work of Donagi and Pantev on Geometric Langlands conjecture using Simpson's 
non-abelian Hodge theory \cite{DonPan19}. 

When the weights are rational, it is known that parabolic bundles can be 
interpreted equivalently as equivariant bundles on a suitable ramified 
Galois cover \cite{Biswas97}, and later, a more intrinsic interpretation was 
found using vector bundles on certain algebraic stacks associated to $(X,D)$, 
namely the  stack of roots \cite{Born07}. More precisely, there is a Fourier-like 
correspondence between parabolic vector bundles on a scheme and ordinary vector 
bundles on certain stack of roots. This naturally raised the question of 
understanding the parabolic connections on smooth varieties through such 
Fourier-like correspondence over root stacks. This has been shown to be 
true over curves \cite{BisMajWan12,LSS13}, and very recently over higher 
dimensions \cite{BornLaar23}. 

Here our aim is to study parabolic connections and similar correspondences 
in the case when the variety $X$ is over complex numbers and possesses a 
real structure, meaning that there is an anti-holomorphic involution on $X$, 
together with a given normal crossings divisor $D$ invariant under the involution. 
Some progress in this topic has been made for curves \cite{Amrutiya14a, Amrutiya14b} using 
suitable ramified Galois covers, but unfortunately it is unclear how to extend 
it for higher dimensions. Let $r$ be an integer. To the data $(r,D)$, we can 
associate a stack of $r$-th roots of $\mc{O}_X(D)$, which we denote by $\mf X$. 
It comes with a natural morphism $\pi: \mf{X}\rightarrow X$. 
We show that, if $X$ possesses a real structure $\sigma$, the stack $\mf X$ 
also possess a real structure $\sigma_{\mf X}$ compatible with that of $X$. 
This is done in Proposition \ref{prop:real-str-on-root-stack}. 
Next, we define the notions of real parabolic bundle and real parabolic connections, 
and prove the real parabolic connections on $X$ (meaning those parabolic connections 
which are compatible with the real structure) correspond to real logarithmic connections 
on these root stacks with their induced real structure. Moreover, under this correspondence, 
the strongly parabolic connections (c.f. Definition \ref{parabolic connection definition}) 
on $X$ correspond to the real holomorphic connections on the root stack. 

We state our main results below. 
Fix an integer $r$, and a reduced effective divisor $D$ satisfying $\sigma^*(D)=D$.

\begin{theorem}[Theorem \ref{thm:real-avatar-of-Biswas-Borne-correspondence}]
	There is an equivalence between the category of real parabolic vector bundles on 
	$(X,\sigma)$ having weights in $\frac{1}{r}\bb Z$ and the category of real vector 
	bundles on $(\mf X,\sigma_{\mf X})$. 
\end{theorem}
If we moreover assume that $D$ is a simple normal crossings divisor, then the pull-back 
of $D$ under the morphism $\pi$ defines a divisor $\mf D$ on $\mf X$ satisfying 
$\pi^*(D) = r\cdot \mf D$. In such cases, we have the following result. 

\begin{theorem}[Theorem \ref{thm:real-log-parbolic-connection-correspondence}]
	There is an equivalence between the category of real logarithmic connections 
	on the corresponding $r$-th root stack $(\mf X,\mf D)$, and the category of 
	real parabolic connections on $(X,D)$ having weights in $\frac{1}{r}\bb Z$. 
	Moreover, under this correspondence, the category of real holomorphic connections 
	on $\mf X$ is equivalent to the category of real strongly parabolic connections on 
	$(X,D)$ having weights in $\frac{1}{r}\bb Z$. 
\end{theorem}

In the final section, we revisit the case when $X$ is a curve, and give a different 
proof for Theorem \ref{thm:real-log-parbolic-connection-correspondence} 
from a different approach using suitable ramified Galois covers on $X$.

\section{Real Structures}\label{sec:real-structures-preliminaries}
Let $j : \bb{C}\rightarrow \mathbb{C}$ be the complex conjugation; 
this is an $\bb R$-algebra isomorphism. 
Let $j^* : \Spec \bb C \to \Spec \bb C$ be the morphism of schemes induced by $j$. 

\begin{definition}
	A {\it real structure} on a $\bb C$-scheme $X$ is an involution 
	$\sigma : X \to X$ lying over $j^*$, i.e., the following diagram commutes: 
	\begin{equation}
	\begin{gathered}
		\xymatrix{
			X \ar[rr]^-\sigma \ar[d] && X \ar[d] \\ 
			\Spec \bb C \ar[rr]^-{j^*} && \Spec \bb C 
		}
	\end{gathered}
	\end{equation} 
\end{definition}

\noindent
Fix an irreducible smooth complex projective variety $X$ 
and a real structure $\sigma$ on it. 

\subsection{Real structures on $\mc O_X$-modules}
Given a sheaf of $\mc{O}_X$-modules $\mathcal{F}$, we denote by $\mathcal{F}^{\sigma}$ 
its direct image sheaf $\sigma_*\mathcal{F}$ considered as an $\mc{O}_X$-module via 
the isomorphism $\sigma^{\#}: \mc{O}_X \to \sigma_*\mc{O}_{X}$ induced by $\sigma$. 
A morphism of $\mc{O}_X$-modules $\varphi:\mathcal{F} \to \mathcal{G}$ 
induces an $\mc{O}_X$-module homomorphism on their direct images 
$\sigma_*\varphi : \sigma_*\mc F \to \sigma_*\mc G$. 

\begin{definition}\label{def:real-structure-on-sheaves-defn}
	\begin{enumerate}[(i)]
		\item A \textit{real structure} on an $\mc{O}_{X}$-module $\mathcal{F}$ 
		is an $\mc O_X$-module isomorphism 
		\begin{equation*}
			\sigma_{\mc F} : \mc F \to \sigma_*\mc F 
		\end{equation*}
		such that $\sigma_*(\sigma_{\mc F})\circ\sigma_{\mc F} = \Id_{\mc F}$. 
		A {\it real $\mc O_X$-module} is a pair $(\mc F, \sigma_{\mc F})$ 
		consisting of an $\mc O_X$-module $\mc F$ and a real structure $\sigma_{\mc F}$ on it. 
		
		\item Let $(\mc F, \sigma_{\mc F})$ and $(\mc G, \sigma_{\mc G})$ be two real $\mc O_X$-modules. 
		An $\mc O_X$-linear map $\varphi : \mc F \to \mc G$ is said to 
		{\it preserve their real structures} if the diagram 
		\begin{align*}
			\xymatrix{
				\mc F \ar[rr]^{\varphi} \ar[d]_{\sigma_{\mc F}} && \mc G \ar[d]^{\sigma_{\mc G}} \\
				\sigma_*\mc F \ar[rr]^{\sigma_*\varphi} && \sigma_*\mc G 
			}
		\end{align*} 
		commutes. 
		
		\item Let $(\mc F, \sigma_{\mc F})$ be a real $\mc O_X$-module. 
		A section $s \in \Gamma(X, \mc F)$ is said to be {\it real} if 
		$\sigma_{\mc F}(X)(s) = s$, where $\sigma_{\mc F}(X) : \Gamma(X, \mc F) \to \Gamma(X, \sigma_*\mc F)$ 
		is the $\bb C$-linear map induced by $\sigma_{\mc F}$. 
	\end{enumerate}
\end{definition}

\begin{remark}
	\begin{enumerate}[(i)]
		\item Note that the isomorphism $\sigma^{\#} : \mc O_X \to \sigma_*\mc O_X$ 
		induced by the real structure $\sigma$ on $X$ is a real structure on $\mc O_X$. 
		
		\item If $E$ and $F$ are $\mc O_X$-modules admitting real structures 
		$\sigma_E$ and $\sigma_F$, respectively, compatible with $\sigma$, 
		then $E\otimes F$ and $E^\vee$ admits real structures $\sigma_E\otimes\sigma_F$ 
		and $\sigma_E^\vee$, respectively, compatible with $\sigma$. 
	\end{enumerate}
	
\end{remark}


\begin{definition}
	A {\it real structure} on a $\bb C$-stack $\mf X$ is a morphism of stacks 
	$F: \mf{X} \to \mf{X}$ such that $F\circ F$ is $2$-isomorphic to the identity 
	morphism $\Id_{\mf{X}}$ of $\mf X$, and the following diagram is $2$-commutative. 
	\begin{align*}
		\xymatrix{
			\mf{X} \ar[d] \ar[r]^{F} & \mf{X} \ar[d]\\
			\Spec\bb{C} \ar[r]^{j^*} & \Spec\bb{C}
		}
	\end{align*}
\end{definition}

\label{sec:real-str-on-a-root-stack}

\section{Real Structures on Root Stacks}
\subsection{Root stack}\label{subsection:root stack}
By an abuse of notation, we denote a scheme and its associated stack 
by the same symbol, if no confusion arises. 
A \textit{generalized Cartier divisor} on a scheme $Y$ is a tuple $(M, s)$, 
where $M$ is an invertible sheaf of $\mc O_Y$-modules on $Y$ and $s \in \Gamma(Y, M)$. 
Given a generalized Cartier divisor $(L, s)$ on $X$ and a positive integer $r$, 
let $\mf X_{(L,\, s,\, r)}$ be the stack 
whose objects are given by tuples $\left(f : U \to X, M, \phi, t\right)$, where 
\begin{itemize}
	\item $f : U \to X$ is a morphism of $\bb C$-schemes, 
	\item $M$ is an invertible sheaf of $\mc O_U$-modules on $U$, 
	\item $t \in \Gamma(U, M)$, and 
	\item $\phi : M^{\otimes r} \stackrel{\simeq}{\longrightarrow} f^*L$ 
	is an $\mc O_U$-module isomorphism such that $\phi(t^{\otimes r}) = f^*s$. 
\end{itemize}
A morphism from 
$\left(f : U \to X, M, \phi, t\right)$ to $\left(f' : U' \to X, M', \phi', t'\right)$ 
in $\mf X_{(L,\, s,\, r)}$ is given by a pair $(g, \psi)$, where 
\begin{itemize}
	\item $g : U \to U'$ is a morphism of $X$-schemes, and 
	\item $\psi : M \to g^* M'$ is an $\mc O_U$-module isomorphism such that $\psi(t)= g^*(t')$. 
\end{itemize}
There is a natural morphism of stacks $\pi : {\mf X}_{(L,\, s,\, r)} \to X$ that sends 
the object $\left(f : U \to X, M, \phi, t\right)$ of $\mf X_{(L,\, s,\,r)}$ 
to the $X$-scheme $f : U \to X$, and a morphism $(g, \psi)$, as above, 
to the morphism of $X$-schemes $g : U \to U'$. 
Note that there is a natural morphism $X \to \Spec\bb C$ 
(arising from the structure morphism of $X$), 
whose composition with $\pi$ makes $\mf{X}_{(L,\, s,\, r)}$ into a $\bb{C}$-stack. 
The stack $\mf X_{(L, s, r)}$ is known as the {\it stack of $r$-th roots of $(L, s)$ over $X$}, 
and it is a Deligne-Mumford stack \cite[Theorem 2.3.3]{Cad07}.


We now construct a real structure on a root stack. 
Let $D$ be a strict normal crossing divisor on $X$, 
meaning that $D$ is a reduced effective divisor on $X$ such that all of its 
irreducible components are smooth and they intersects each other transversally. 
Recall that the invertible sheaf $\mc O_X(D)$ is a subsheaf of $\mc K_X$, 
where $\mc K_X$ is the constant sheaf given by $K(X)$, the function field of $X$. 
The constant function $1 \in K(X)$ gives rise to a global section of $\mc O_X(D)$, 
which we denote by $s_D$. Thus we get a generalized Cartier divisor $(\mc O_X(D),\, s_D)$ 
on $X$. Note that $D$ is the divisor of zeroes of $s_D$.

We say that $D$ is a \textit{real} divisor if it is invariant under $\sigma$ in the sense that 
$\sigma(D) = D$. Since $\sigma\circ\sigma = \Id_X$, it follows that $D$ is real if and only if 
$\sigma^*(D) := \sigma^{-1}(D) = D$. Henceforth we assume that $D$ is real. 

\begin{proposition}\label{prop:real-str-on-O_X(D)}
	With the above notations, the line bundle $\mc O_X(D)$ admits a real structure 
	for which $s_D$ is a real section if and only if $D$ is a real divisor. 
\end{proposition}

\begin{proof}
	Let $L := \mc O_X(D)$. 
	Since $\sigma$ is an involution of $X$, we have a canonical 
	$\mc{O}_X$-linear isomorphism $\sigma^*\sigma_*L\simeq L$.  
	Applying $\sigma^*$ on both sides, we have an isomorphism 
	\begin{equation}\label{iso1}
		p : \sigma_*L \longrightarrow \sigma^*L. 
	\end{equation} 
	The condition $\sigma(D) = D$ gives rise to an isomorphism 
	\begin{equation}\label{iso2}
		h : L \longrightarrow \sigma^*L.  
	\end{equation}
	Thus 
	$$\sigma_L := p^{-1}\circ\,h : L \longrightarrow \sigma_*L$$
	is a real structure on $L$. 
	Note that the canonical section $s_D \in H^0(X, \mc O_X(D))$ 
	remains fixed under $\sigma_L$, and hence is a real section of $L = \mc O_X(D)$. 
\end{proof}

\begin{remark}\label{rem:global section}
	Note that, the map 
	$$H^0(p) : \Gamma(X,\sigma_*L) = \Gamma(X,L) \longrightarrow \Gamma(X,\sigma^*L)$$
	induced by $p$ at the level of global sections is given by $s\mapsto \sigma^*s$. 
\end{remark}

\begin{proposition}\label{real structure on root stack proposition}
	Fix an integer $r \geq 2$, and a real strict normal crossing divisor 
	$D$ on $(X, \sigma)$. Let 
	$$\pi : \mf X := \mf X_{(\mc O_X(D),\,s_D,\,r)} \longrightarrow X$$ 
	be the stack of $r$-th roots of the generalized Cartier divisor 
	$(\mc O_X(D), s_D)$ on $X$. Then $\mf X$ admits a real structure 
	$\sigma_{\mf{X}} : \mf X \to \mf X$ compatible with the real 
	structure $\sigma$ on $X$ in the sense that the following diagram 
	\begin{equation}\label{eqn:real-str-on-root-stack}
		\begin{gathered}
			\xymatrix{
				\mf{X} \ar[rr]^{\sigma_{\mf{X}}} \ar[d]_{\pi} && \mf{X} \ar[d]^{\pi} \\
				X \ar[rr]^{\sigma} && X 
			}
		\end{gathered}
	\end{equation}
	is $2$-commutative and the two natural transformations 
	$\pi\circ \sigma_{\mf{X}}$ and $\sigma\circ \pi$ are naturally isomorphic 
	(in fact, they are equal). 
\end{proposition}

\begin{proof}
	Note that a morphism of stacks 
	$\sigma_{\mf X} : \mf X \longrightarrow \mf X$, 
	such that $\sigma_{\mf X}\circ\sigma_{\mf X}$ is $2$-isomorphic to $\Id_{\mf X}$ 
	and the diagram \eqref{eqn:real-str-on-root-stack} is $2$-commutative, 
	is a real structure on $\mf X$ (c.f. Definition \ref{def:real-str-on-a-stack}). 
	Given a scheme $U$ and an object 
	$$\tau := (f : U \to X, M, \phi, t) \in \mf X(U),$$ 
	we define $\sigma_{\mf X}(\tau) \in \mf X(U)$ as follows. 
	Recall that $L := {\mc O}_X(D)$ has a real structure $\sigma_L$ 
	induced by the canonical isomorphism $p$ in \eqref{iso1}. 
	Since $\phi: M^{\otimes r} \stackrel{\simeq}{\longrightarrow} f^*L$ 
	by assumption, pulling back the isomorphism 
	$h : L \stackrel{\simeq}{\longrightarrow} \sigma^*L$ in \eqref{iso2} 
	along $f$ and composing it with $\phi$ we obtain an isomorphism of $\mc O_U$-modules 
	\begin{equation}
		(f^*h)\circ\phi : 
		M^{\otimes r} \overunderset{\phi}{}{\longrightarrow} 
		f^*L \overunderset{f^*h}{}{\longrightarrow} 
		f^*(\sigma^*L) = (\sigma\circ f)^*L,  
	\end{equation}
	which, at the level of global sections, sends $t^{\otimes r}$ to 
	\begin{align*}
		\left((f^*h)\circ \phi\right)(t^{\otimes r}) &= (f^*h)(f^*s_D) \\ 
		&= f^*(h(s_D))\\
		&= f^*\left((p\circ \sigma_L)(s_D)\right)\\
		&= f^*\left(p(s_D)\right), \ \ \ [\text{since }\, \sigma_L(s_D)=s_D.]\\
		&= f^*(\sigma^*(s_D)), \ \ \ [\text{c.f. Remark \ref{rem:global section}.}] \\
		&= (\sigma\circ f)^*(s_D)\,.
	\end{align*}
	Thus, we define $\sigma_{\mf X}(U)(\tau)$ to be the object 
	\begin{align}\label{real structure mapping}
		\sigma_{\mf{X}}(U)(\tau) := \left(\sigma\circ f: U\rightarrow X\,,\, M\,,\, (f^*h)\circ \phi\,,\,t\right)\, \in \mf X(U).
	\end{align}
	To define the action of $\sigma_{\mf X}$ on the arrows, 
	recall that an arrow from $\tau := \left(f : U \to X, M, \phi, t\right)$ 
	to $\tau' := \left(f' : U' \to X, M', \phi', t'\right)$ in $\mf X$ 
	is given by a pair $(g, \psi)$, where 
	\begin{itemize}
		\item $g : U \to U'$ is a morphism of schemes such that $f'\circ g = f$, and 
		\item $\psi : M \stackrel{\simeq}{\longrightarrow} g^* M'$ is an isomorphism 
		of $\mc O_U$-modules such that $\psi(t)= g^*(t')$.
	\end{itemize}
	Since $(\sigma\circ f')\circ g = \sigma\circ f$, 
	it immediately follows that the pair $(g, \psi)$ defines an arrow from 
	$\sigma_{\mf X}(\tau)$ to $\sigma_{\mf X}(\tau')$ in $\mf X$, as required. 
	It is straight-forward to check that $\sigma_{\mf X}$ preserves composition of arrows. 
	Thus, $\sigma_{\mf X} : {\mf X} \longrightarrow {\mf X}$ is a morphism of stacks. 
	Moreover, the composition $\sigma_{\mf X}\circ\sigma_{\mf X}$ takes $\tau$ to the quadruple 
	\begin{align*}
		\left(\sigma\circ(\sigma\circ f)\,,\, M\,,\, ((\sigma\circ f)^*h)\circ(f^*h\circ \phi)\,,\,t\right) 
		&= \left(f\,,\,M\,,\,(f^*\sigma^*h)\circ f^*h\circ \phi\,,\, t\right) \\
		&= \left(f\,,\,M\,,\,f^*(\sigma^*h\circ h)\circ \phi\,,\,t\right) \\
		&= \left(f\,,\,M\,,\,f^*(\textnormal{Id}_L)\circ\phi\,,\,t\right), \,\,\,
	\end{align*}
	where the last equality $\sigma^*h\circ h = \Id_L$ follows from the 
	canonical identification $\sigma^*(\sigma^*L)\cong L$. 
	Then it follows immediately that 
	$\sigma_{\mf X}\circ \sigma_{\mf X} = \Id_{\mf X}$. 
	Moreover, $\sigma_{\mf X}$ lies over $j^* : \Spec\bb{C} \to \Spec\bb{C}$, 
	since composing $f : U \rightarrow X$ with $\sigma : X \rightarrow X$ 
	changes the structure morphism $U \rightarrow \Spec\bb{C}$ of $U$ by $j^*$. 
	Consider $\tau = \left(f : U \to X, M, \phi, t\right) \in \mf X(U)$ as above. 
	Note that, both $(\pi\circ\sigma_{\mf X})(\tau)$ and $(\sigma\circ\pi)(\tau)$ 
	coincide with $\sigma\circ f : U \to X$. 
	If $\tau' := \left(f': U'\to X, M', \phi',t' \right)$ and 
	$(g, \psi) \in {\rm Hom}_{\mf X}(\tau, \tau')$ is an arrow in $\mf X$ as above, 
	it directly follows from the above discussion that both $\pi\circ\sigma_{\mf X}$ 
	and $\sigma\circ\pi$ maps $(g, \psi)$ to $g : U \to U'$. Hence the result follows. 
\end{proof}

The next lemma describes pullback of atlas of $\mf X$ along its real structure. 

\begin{lemma}\label{lem:pullback-of-atlas-under-real-str-of-root-stack}
	Let $U$ be a scheme. Let $u: U \to \mf X$ be the morphism of stacks given by the quadruple 
	$(f : U \to X, M, \phi, t) \in \mf X(U)$, where $f = \pi\circ u$, and let 
	\begin{equation}
		\begin{gathered}
			\xymatrix{
				\sigma^{-1}U := X\times_{\sigma,\,X,\,f}U 
				\ar[rr]^-{\sigma_U} \ar[d]_-{f_{\sigma}} 
				&& U \ar[d]^-f \\ 
				X \ar[rr]^-{\sigma} && X 
			}
		\end{gathered} \nonumber 
	\end{equation}
	be the associated cartesian diagram in schemes. 
	Then there is a morphism 
	\begin{equation}\label{eqn:sigma-inv-U-valued-point-of-root-stack-X}
		u_{\sigma_{\mf X}} : \sigma^{-1}U \longrightarrow \mf X 
	\end{equation} 
	such that $\pi\circ u_{\sigma_{\mf X}} = f_{\sigma}$ and the following diagram 
	is cartesian. 
	\begin{equation}\label{diag:pullback-of-valued-point-of-X-along-real-structure}
		\begin{gathered}
			\xymatrix{
				\sigma^{-1}U \ar[rr]^-{\sigma_U} \ar[d]_-{u_{\sigma_{\mf X}}} 
				&& U \ar[d]^-{u} \\ 
				\mf X \ar[rr]^-{\sigma_{\mf X}} && \mf X 
			}
		\end{gathered}
	\end{equation}
	In particular, $\mf X\times_{\sigma_{\mf X}, \mf X, u} U \cong \sigma^{-1}U$ as $\bb C$-schemes. 
	%
		%
\end{lemma}

\begin{proof}
	Recall that the real structure on $L =\mc O_X(D)$ induces an isomorphism of line bundles 
	$\widehat{\sigma_L} : \sigma^*L \longrightarrow L$ 
	on $X$, whose pullback along $f_{\sigma}$ gives an isomorphism of invertible sheaves 
	$f_{\sigma}^*(\widehat{\sigma_L}) : f_{\sigma}^*\sigma^*L \longrightarrow f_{\sigma}^*L$. 
	Then the composite map 
	\begin{equation}\label{def:pullback-of-isom-phi-along-real-structure-of-X}
		\phi_{\sigma} 
		: \sigma_U^*M^{\otimes r} \stackrel{\sigma^*_U\phi}{\longrightarrow} 
		f_\sigma^*\sigma^*L \stackrel{f_{\sigma}^*(\widehat{\sigma_L})}{\longrightarrow} 
		f_{\sigma}^*L 
	\end{equation}
	is an isomorphism of invertible sheaves on $\sigma^{-1}U$. 
	We claim that $\phi_{\sigma}$ sends the section 
	$\sigma_U^*t^{\otimes r}$ to $f_{\sigma}^*s_D$, where $s_D$ is the 
	canonical section of $L = \mc O_X(D)$. 
	Note that, at the level of global sections $\widehat{\sigma_L}$ 
	is given by the composition 
	\begin{align*}
		\Gamma(X,\sigma^*L) & \longrightarrow \Gamma(X,\sigma^*\sigma_*L) 
		\longrightarrow \Gamma(X,L)\\ 
		\sigma^*(s_D) & \longmapsto \sigma^*(\widehat{\sigma_L}(s_D)) 
		\ \ \longmapsto\ \ \sigma_L(s_D). 
	\end{align*}
	Therefore, we have 
	\begin{align*}
		\phi_{\sigma}(\sigma_U^*t^{\otimes r}) 
		&= \left(f_{\sigma}^*(\widehat{\sigma_L})\circ(\sigma^*_U\phi)\right)
		(\sigma_U^*t^{\otimes r}) \\ 
		&= f_{\sigma}^*(\widehat{\sigma_L})\sigma_U^*\big(\phi(t^{\otimes r})\big) \\ 
		&= f_{\sigma}^*(\widehat{\sigma_L})\sigma_U^*(f^*s_D), 
		\ \ \ \text{since } \phi(t^{\otimes r}) = f^*s_D. \\ 
		&= f_{\sigma}^*(\widehat{\sigma_L}) f_\sigma^*\sigma^*(s_D) \\ 
		&= f_\sigma^*\left(\widehat{\sigma_L}(\sigma^* s_D)\right) \\ 
		&= f_\sigma^*\left(\sigma_L(s_D)\right) \\ 
		&= f_\sigma^*(s_D), \ \ \ \text{since } \ \sigma_L(s_D) = s_D. 
	\end{align*}
	Therefore, the quadruple 
	\begin{equation}\label{def:pullback-of-atlas-along-real-str-of-root-stack}
		\sigma^*u := 
		\left(f_{\sigma} : \sigma^{-1}U \to X,\, \sigma_U^{-1}M,\, 
		\phi_{\sigma},\, \sigma_U^*(t)\right) 
	\end{equation}
	is an object of $\mf X(\sigma^{-1}(U))$ which gives a morphism 
	$u_{\sigma_{\mf X}} : \sigma^{-1}U  \to \mf X$ such that 
	$\pi\circ u_{\sigma_{\mf X}} = f_\sigma$. 
	
	It remains to show that the diagram in 
	\eqref{diag:pullback-of-valued-point-of-X-along-real-structure} is cartesian. 
	First of all, since $\sigma_{\mf X}$ is an isomorphism, it is representable 
	and the fiber product $\mf X\times_{\sigma_{\mf X}, \mf X, u} U$ 
	is a scheme. 
	Given any scheme $T$, the objects of 
	$(\mf X\times_{\sigma_{\mf X}, \mf X, u} U)(T)$ 
	are triples $(x, g, \gamma)$, where $x \in \mf X(T)$, $g \in U(T) = \Hom(U, T)$ 
	and $\gamma : u(g) \longrightarrow \sigma_{\mf X}(x)$ is an isomorphism in $\mf X(T)$. 
	Note that, 
	$u(g)$ is given by the quadruple $(f\circ g : T\rightarrow X, g^*M, g^*\phi, g^*t) \in \mf X(T)$. 
	If $x$ is given by the quadruple $(\alpha : T \to X, M', \phi', t') \in \mf X(T)$, then 
	$\sigma_{\mf X}(x)$ is given by the quadruple 
	$(\sigma\circ\alpha : T \to X, M', (f^*h)\circ\phi', t')$ (see \eqref{real structure mapping}). 
	Note that the first component of $\sigma_{\mf X}(x)$ is $\sigma\circ\alpha : T \rightarrow X$ by definition. 
	Since morphisms in ${\mf X}(T)$ 
	are lying over the identity morphism $\Id_T$ of $T$, 
	this immediately implies 
	that the diagram 
	\begin{align*}
		\xymatrix{
			T \ar[rr]^{g} \ar[d]_{\alpha} && U \ar[d]^{f} \\
			X \ar[rr]^{\sigma} && X
		}
	\end{align*}
	commutes. 
	Then by universal property of fiber product of morphisms of schemes, 
	both $g$ and $\alpha$ factor through a unique morphism $T \rightarrow \sigma^*U$. 
	This gives a bijection 
	\begin{align}
		\left(U\times_{\sigma_{\mf X},\,\mf X,\,u} U\right)(T) \stackrel{\simeq}{\longrightarrow} (\sigma^*U)(T) 
	\end{align}
	which is functorial in $T$, and hence the diagram 
	\eqref{diag:pullback-of-valued-point-of-X-along-real-structure} is cartesian. 
	It follows that $\mf X\times_{\sigma_{\mf X}, \mf X, u} U$ is isomorphic to $\sigma^{-1}U$ as $\bb C$-schemes. 
\end{proof}

\subsection{Real structure on the tautological line bundle on root stack}\label{sec:real-str-on-the-tautological-line-bundle}
Recall that 
if $\varphi : \mc{X} \rightarrow \mc{Y}$ is a representable morphism of 
Deligne-Mumford stacks, the direct image of a coherent sheaf $\mc F$ of 
$\mc{O}_{\mc X}$-modules along $\varphi$ can be described over \'etale 
atlases as follows: 
since $\varphi$ is a representable morphism of DM stacks, 
given any \'etale atlas $\psi : U \rightarrow {\mc Y}$ of $\mc Y$, 
the fiber product 
$$
\xymatrix{
\varphi^*U := \mc X\times_{\mc Y}U 
\ar[rr]^-{\psi_{\mc X}} && \mc X 
}
$$ 
is again an \'etale atlas of $\mc X$. 
\begin{equation}
	\begin{gathered}
		\xymatrix{
			\mc X\times_{\mc Y} U \ar[rr]^{\varphi_U} \ar[d]_{\psi_{\mc X}} && U \ar[d]^{\psi} \\
			\mc{X} \ar[rr]^{\varphi} && \mathcal{Y}
		} 
		\nonumber 
	\end{gathered}
\end{equation} 
Then $\mc F$ is represented by a coherent sheaf $F$ 
on $\mc X\times_{\mc Y}U$, 
and we define the direct image of $\mc F$ along $\varphi$ to be the sheaf of 
$\mc{O}_{\mc Y}$-modules $\varphi_*\mc{F}$ defined by $(\varphi_U)_*F$ on $(U, \psi)$. 
The compatibility conditions for different choices of atlases can be easily checked 
(c.f. \cite[\S\,9]{Ols16}). 

\begin{definition}
	Let $\mc X$ be an algebraic stack over $\bb C$ admitting a 
	real structure $\sigma_{\mc X}$. 
	Let $\mc F$ be a sheaf of $\mc O_{\mc X}$-modules on $\mc X$. 
	A \textit{real structure} on $\mc F$ is an ${\mc O}_{\mc X}$-module isomorphism 
	\begin{equation}
		\sigma_{\mc F} : \mc F \to (\sigma_{\mc X})_*\mc F \nonumber 
	\end{equation}
	such that 
	$\left(\big(\sigma_{\mc X}\big)_*\sigma_{\mc F}\right)\circ\sigma_{\mc F} 
	= \Id_{\mc F}$ (c.f. Definition \ref{def:real-structure-on-sheaves-defn}). 
\end{definition}



Continuing with the same notations and assumptions from the last subsection, 
let $(\ms M, \xi)$ be the tautological generalized Cartier divisor on the 
root stack $\mf X = \mf X_{(L,\, s_D,\, r)}$, where $L := \mc O_X(D)$.  

\begin{proposition}\label{prop:real-structure-on-tautological-line-bundle}
	With the above notations, 
	the tautological invertible sheaf $\ms{M}$ on $\mf{X}$ admits a real structure 
	$\sigma_{\ms M}$ coming from the real structure $\sigma_{\mf X}$ on ${\mf X}$, 
	and that $\xi$ is a real section. Consequently, $(\ms M, \xi)$ admits a real structure. 
\end{proposition}

\begin{proof}
	Let $u : U \rightarrow \mf{X}$ be an \'etale atlas of $\mf X$ defined by the quadruple 
	$(f : U \to X, M, \phi, t)$, so that $\ms{M}_u := u^*\ms M = M$ on $U$. 
	It follows from Lemma \ref{lem:pullback-of-atlas-under-real-str-of-root-stack} 
	that the fiber product $U\times_{\sigma_{\mf X},\,\mf{X},\,u} U \cong \sigma^*U$ 
	and the invertible sheaf $u_{\sigma_{\mf X}}^*\ms{M}$ on the pull-back atlas 
	$u_{\sigma_{\mf X}} : \sigma^*U \rightarrow \mf{X}$ is given by $\sigma_U^*M$. 
	Then it follows that  
	$$\left((\sigma_{\mf X})_*\ms{M}\right)_u = (\sigma_U)_*(\ms{M}_{u'}) = (\sigma_U)_*(\sigma_U^*M).$$ 
	Since $\sigma: X \rightarrow X$ is an isomorphism, so is the pull-back map 
	$\sigma_U : \sigma^*U \rightarrow U$. Then we have an isomorphism of invertible sheaves 
	$M\simeq(\sigma_U)_*(\sigma_U^*M)$, which gives an isomorphism 
	$\ms{M}_u \simeq \left((\sigma_{\mf{X}})_*\ms{M}\right)_u$ 
	over the \'etale atlas $u : U \to \mf X$ of $\mf X$. 
	This produces a real structure 
	\begin{equation}\label{def:real-str-on-tautological-invertible-sheaf-on-root-stack}
		\sigma_{\ms M} : \ms{M} \stackrel{\simeq}{\longrightarrow} (\sigma_{\mf{X}})_*\ms{M} 
	\end{equation}
	on $\ms{M}$, as required. The second part follows.  
\end{proof}

%

As before, let $L := \mc O_X(D)$, and 
let $u : U \rightarrow \mf X$ be an \'etale atlas of $\mf X := \mf X_{(L,\, s_D,\, r)}$ 
given by the quadruple $(f : U \rightarrow X, M, \phi, t) \in \mf X(U)$. 
The isomorphism $\phi : M^{\otimes r} \longrightarrow f^*L$ 
give rise to an isomorphism of invertible sheaves 
\begin{equation}\label{eqn:canonical-isom-of-M^r-with-pullback-of-O_X(D)}
	\psi : {\ms M}^{\otimes r} \longrightarrow \pi^*L
\end{equation}
on $\mf X$. 
The next lemma says that the above isomorphism is compatible with their respective 
real structures $\sigma_{{\ms M}^{\otimes r}}$ and $\pi^*(\sigma_{L})$. 

\begin{lemma}\label{real structure preserving isomorphism}
	With the above notations, the canonical isomorphism of invertible sheaves 
	in \eqref{eqn:canonical-isom-of-M^r-with-pullback-of-O_X(D)} is compatible 
	with the real structures $\sigma_{{\ms M}^{\otimes r}}$ and $\pi^*(\sigma_{L})$ 
	on ${\ms M}^{\otimes r}$ and $\pi^*L$, respectively. 
\end{lemma}

\begin{proof}
	Let $u : U \rightarrow \mf X$ be an \'etale atlas of $\mf X$ given by the quadruple 
	$(f : U \rightarrow X, M, \phi, t)$. Note that the pullback of the isomorphism 
	$$\psi : {\ms M}^{\otimes r} \longrightarrow \pi^*L$$ 
	along $u$ is the isomorphism $$\phi : M^{\otimes r} \longrightarrow f^*L,$$ 
	whereas the pullback of the isomorphism 
	$(\sigma_{\mf X})_*\psi : (\sigma_{\mf X})_*\ms M^{\otimes r} 
	\longrightarrow (\sigma_{\mf X})_*(\pi^*L)$ 
	along $u$ is the isomorphism 
	$$(\sigma_U)_*\left((f_\sigma^*\widehat{\sigma_L})\circ(\sigma_U^*\phi)\right) : 
	(\sigma_U)_*(\sigma_U^*M^{\otimes r}) \longrightarrow (\sigma_U)_*(f_\sigma^*L),$$ 
	(see \eqref{def:pullback-of-isom-phi-along-real-structure-of-X} and 
	\eqref{def:pullback-of-atlas-along-real-str-of-root-stack} 
	in Lemma \ref{lem:pullback-of-atlas-under-real-str-of-root-stack}). 
	Note that
	\begin{align}
		(\sigma_U)_*((f_\sigma^*\widehat{\sigma_L})\circ(\sigma_U^*\phi)) 
		= (\sigma_U)_*(f_\sigma^*\widehat{\sigma_L})\circ(\sigma_U)_*(\sigma_U^*\phi).
	\end{align}
	To show that $\psi$ is compatible with the real structures under consideration, 
	we need to show that the following diagram commutes. 
	\begin{equation}
		\begin{gathered}
			\xymatrix{
				\ms M^{\otimes r} \ar[dd]^{\psi} \ar[rr]^-{\sigma_{\ms M^{\otimes r}}} 
				&& \left(\sigma_{\mf X}\right)_*(\ms M^{\otimes r}) 
				\ar[dd]^{\left(\sigma_{\mf X}\right)_*\psi} \\
				&& \\
				\pi^*L \ar[rr]^-{\pi^*(\sigma_L)} && \left(\sigma_{\mf X}\right)_*(\pi^*L) 
			}
		\end{gathered}
	\end{equation}
	Note that the pullback of this diagram along $u$ coincides with the 
	outermost square of the following diagram: 
	\begin{equation*}
		\begin{gathered}
			\xymatrix{
				M^{\otimes r} \ar[rr]^-{\simeq} \ar[dd]_{\phi} && (\sigma_U)_*(\sigma_U^*M^{\otimes r}) 
				\ar[dd]_{(\sigma_U)_*(\sigma_U^*\phi)} \ar[rrr]^{\Id} &&& (\sigma_U)_*(\sigma_U^*M^{\otimes r}) 
				\ar[dd]^{(\sigma_U)_*((f_\sigma^*\widehat{\sigma_L})\circ(\sigma_U^*\phi))} \\
				&&&&& \\
				f^*L \ar[rr]^-{\simeq} && (\sigma_U)_*(\sigma_U^*f^*L) 
				\ar[rrr]^{(\sigma_U)_*(f_\sigma^*\widehat{\sigma_L})} &&& (\sigma_U)_*(f_\sigma^*L) 
			}
		\end{gathered}
	\end{equation*}
	Since the leftmost and rightmost squares of this diagram commutes, 
	we conclude that the outermost rectangle also commutes. 
	This proves our claim. 
\end{proof}


\section{Biswas-Borne Correspondence} 

\subsection{Real structures on parabolic vector bundles}
We recall the definition of parabolic vector bundles 
\cite{MY92, Yok95, Biswas97, Born07}.  
Fix a positive integer $r$, and consider the ordered set $\frac{1}{r}\bb{Z}$ as a category, 
where there is an arrow $a \rightarrow b$ in $\frac{1}{r}\bb{Z}$ if and only if $a \leq b$. 
Let $\Vect(X)$ be the category of algebraic vector bundles on $X$. 

\begin{definition}\label{def:parabolic-vector-bundle}
	A {\it parabolic vector bundle} on $X$ {\it with parabolic structure along} $D$ and 
	{\it rational weights in $\frac{1}{r}\bb Z$} is a pair $(E_\bullet, {\mf j}_{E_\bullet})$, 
	where 
	\begin{equation*}
		E_\bullet : \left(\frac{1}{r}\bb{Z}\right)^{\rm op} \longrightarrow \Vect(X)
	\end{equation*}
	is a functor and 
	$${\mf j}_{E_\bullet} : E_\bullet\otimes{\mc O}_X(-D) 
	\stackrel{\simeq}{\longrightarrow} E_{\bullet+1}$$ 
	is a natural isomorphism of functors such that the following diagram commutes: 
	\begin{equation}\label{diag:parabolic-bundle-definition}
		\begin{gathered}
			\xymatrix{
				E_\bullet\otimes\mc O_X(-D) \ar[rr]^-{{\mf j}_{E_\bullet}} 
				\ar[dr]_-{} && E_{\bullet+1} \ar[dl]^-{} \\ 
				& E_\bullet & 
			}
		\end{gathered}
	\end{equation}
	We drop the symbol ${\mf j}_{E_\bullet}$ from $(E_\bullet, {\mf j}_{E_\bullet})$, 
	and simply write it as $E_\bullet$ unless we need its reference, 
	and call them simply as {\it parabolic bundles}, 
	whenever the data of $D$ and $r$ are clear from the context. 
	A morphism of parabolic bundles from $E_{\bullet}$ and $F_{\bullet}$ is a 
	natural transformation of functors compatible with the diagram 
	\eqref{diag:parabolic-bundle-definition} in the obvious sense. 
	We denote by $\ParVect_{\frac{1}{r}}(X, D)$ the category of parabolic vector bundles 
	on $X$ with parabolic structure along $D$ and having rational weights in $\frac{1}{r}\bb Z$. 
\end{definition}
Next we define the notion of a real structure on a parabolic vector bundles. 
For this, note that the direct image $\sigma_*$ defines a functor in an obvious way, 
also denote by the same symbol: 
\begin{equation}
	\sigma_* : \Vect(X) \longrightarrow \Vect(X). 
\end{equation}

\noindent
As in \S\ref{sec:real-structures-preliminaries}, 
henceforth we assume the $X$ is equipped with a real structure $\sigma$ with $\sigma(D) = D$. 

\begin{definition}\label{real def:parabolic-vector-bundle}
	A {\it real structure} on a parabolic vector bundle $E_\bullet \in \ParVect_{\frac{1}{r}}(X, D)$ 
	is a natural isomorphism of functors 
	\begin{equation}
		\sigma_{E_\bullet} : E_\bullet \longrightarrow \sigma_*E_\bullet 
	\end{equation}
	such that $\sigma_*(\sigma_{E_\bullet})\circ\sigma_{E_\bullet} = \Id_{E_\bullet}$, 
	where $\sigma_*(\sigma_{E_\bullet})$ and its direct image under $\sigma$ have 
	their usual meaning as in the last section. 
	A {\it real parabolic vector bundle} on $(X, \sigma, D)$ is a pair 
	$(E_\bullet, \sigma_{E_\bullet})$ consisting of a parabolic bundle 
	$E_\bullet \in \ParVect_{\frac{1}{r}}(X, D)$ and a real structure 
	$\sigma_{E_\bullet}$ on it. 
\end{definition}
Note that a real structure $\sigma_{E_\bullet}$ on $(E_\bullet, \mf j_{E_\bullet})$ 
induces a real structure $\sigma_{E_{\ell/r}}$ on $E_{\ell/r}$, for each $\ell \in \bb Z$, 
which makes the following diagram commutative: 
\begin{equation}
	\begin{gathered}
		\xymatrix{
			E_{\frac{\ell}{r}}\otimes{\mc O}_X(-D) 
			\ar[rrr]^{{\sigma}_{E_{\frac{\ell}{r}}}\,\otimes\,\sigma_{{\mc O}_X(-D)}} 
			\ar[d]_-{\mf j_{E_{\ell/r}}} &&& 
			\sigma_*\left(E_{\frac{\ell}{r}}\otimes_{\mc{O}_X}\mc{O}_X(-D)\right) 
			\ar[d]^-{\sigma_*\left(\mf j_{E_{\ell/r}}\right)} \\ 
			E_{\frac{\ell}{r}+1} \ar[rrr]^{{\sigma}_{(E_{\frac{\ell}{r}+1})}} 
			&&& \sigma_*\left(E_{\frac{\ell}{r}+1}\right)\,, 
		}
	\end{gathered}
\end{equation}
where $\sigma_{\mc O_X(-D)}$ is the canonical real structure on ${\mc O}_X(-D)$.  
A morphism of real parabolic bundles $(E_\bullet, \sigma_{E_\bullet}) \to (F_\bullet, \sigma_{F_\bullet})$ is given by a morphism $\varphi : E_\bullet \to F_\bullet$ 
in $\ParVect_{\frac{1}{r}}(X, D)$ such that 
\begin{equation}
	\sigma_*(\varphi)\circ\sigma_{E_\bullet} = \sigma_{F_\bullet}\circ\varphi. 
\end{equation}
We denote by $\RPVect_{\frac{1}{r}}(X,D)$ the category of real parabolic bundles 
on $(X, \sigma, D)$. 


\subsection{Correspondence with real structure}
Consider the root stack $\mf{X} := \mf{X}_{(\mc{O}_X(D),\,s_D,\,r)}$ 
together with the tautological invertible sheaf $\ms{M}$ on it 
(see \S\ref{sec:real-structures-preliminaries}).  
Let $\pi : \mf{X} \longrightarrow X$ be the natural morphism. 
In \cite{Biswas97}, it is shown that the category of parabolic bundles 
on $X$ with parabolic weights in $\frac{1}{r}\bb Z$ along $D$ is equivalent 
to the category of orbifold bundles on a suitable ramified Galois cover 
of $X$; this is known as Biswas correspondence. 
The following result is the stacky avatar of this correspondence.  

\begin{theorem}\cite[Th\'eor\'eme 2.4.7]{Born07}\label{thm:Biswas-Borne-correspondence} 
	There is an equivalence of categories between $\Vect(\mf{X})$ and $\ParVect_{\frac{1}{r}}(X,D)$ 
	given by the following correspondence:
	\begin{enumerate}[$\bullet$]
		\item A vector bundle $\mathcal{E}$ on $\mf{X}$ gives rise to the 
		parabolic vector bundle $E_{\bullet} \in \ParVect_{\frac{1}{r}}(X, D)$ 
		given by 
		$$\frac{\ell}{r} \longmapsto E_{\frac{\ell}{r}} 
		:= \pi_*\left(\mathcal{E}\otimes_{\mc O_{\mf X}}\ms M^{-\ell}\right),\,\,\forall\,\ell\,\in \bb Z.$$ 
		
		\item Conversely, a parabolic vector bundle $E_{\bullet} \in \ParVect_{\frac{1}{r}}(X, D)$ 
		gives rise to the vector bundle $\mathcal{E}$ on $\mf{X}$ given by 
		$$\mathcal{E} := \int^{\frac{1}{r}\bb{Z}} 
		\pi^*(E_{\bullet})\otimes_{{\mc O}_{\mf X}} \ms{M}^{\bullet\,r},$$
		where $\int^{\frac{1}{r}\mathbb{Z}}$ stands for the coend.
	\end{enumerate}
\end{theorem}

Recall that $\mf X$ admits a real structure $\sigma_{\mf X}$ compatible with the 
real structure $\sigma$ on $X$ (see Proposition \ref{real structure on root stack proposition}), 
which induces a real structure $\sigma_{\ms{M}}$ on the tautological invertible sheaf 
$\ms M$ on $\mf{X}$ by Proposition \ref{prop:real-structure-on-tautological-line-bundle}. 
The following theorem is a real avatar of Theorem \ref{thm:Biswas-Borne-correspondence}.

\begin{theorem}\label{thm:real-avatar-of-Biswas-Borne-correspondence}
	The categories ${\bf RealVect}(\mf{X})$ and ${\bf RealParVect}_{\frac{1}{r}}(X,D)$ 
	are equivalent. 
\end{theorem}

Before going into the proof of Theorem \ref{thm:real-avatar-of-Biswas-Borne-correspondence}, 
we make the following remark whose proof is a straightforward application of the 
universal property of colimit, so we omit it. 

\begin{remark}\label{rem:isomorphism-of-colimits}
	Let $J$ be an indexing category, and let $\ms C$ be another category. 
	Let $F, G: J \rightarrow \ms C$ be two functors such that both the colimits 
	$\underset{j \in J}{\lim} F_j$ and $\underset{j \in J}{\lim} G_j$ exist. 
	Then any natural isomorphism of functors 
	$\varphi := \{\varphi_j\}_{j\in J} : F \longrightarrow G$ uniquely extends 
	to a natural isomorphism of functors 
	$\widetilde{\varphi} : \underset{j \in J}{\lim} F_j 
	\longrightarrow \underset{j \in J}{\lim} G_j$ 
	such that for each $j \in J$ the following diagram commutes:
	\begin{align*}
		\xymatrix{
			F_j \ar[rr] \ar[d]_{\varphi_j} && 
			\underset{j \in J}{\lim} F_j \ar[d]^{\widetilde\varphi} \\
			G_j \ar[rr] && \underset{j \in J}{\lim} G_j\,. 
		}
	\end{align*}
\end{remark}

\begin{proof}[Proof of Theorem \ref{thm:real-avatar-of-Biswas-Borne-correspondence}] 
	Let $(\mc{E}, \sigma_{\mc E})$ be a real vector bundle on $\mf{X}$. 
	We show that the corresponding parabolic vector bundle $E_\bullet$ has a real structure. 
	Applying $\pi_*$ to $\sigma_{\mc E}$ 
	we get a real structure on $\pi_*\mc E$, namely 
	\begin{equation}\label{real structure equation}
		\pi_*(\sigma_{\mc E}) : \pi_*\mc{E} \longrightarrow 
		\pi_*\left((\sigma_{\mf X})_*\mc{E}\right) \cong 
		\sigma_*\left(\pi_*\mc{E}\right), 
	\end{equation}
	where the last isomorphism follows from Proposition 
	\ref{real structure on root stack proposition}. 
	Since each $\mc{E}\otimes_{\mc O_{\mf X}}\ms{M}^{-\ell}$ has a real structure, 
	namely $\sigma_{\mc E}\otimes\sigma_{{\ms M}^{-\ell}}$, for each $\ell \in \bb Z$, 
	we get a real structure on 
	$E_{\frac{\ell}{r}} := \pi_*(\mc E\otimes_{\mc{O}_{\mf X}}{\ms M}^{-\ell})$, 
	which we denote by $\sigma_{E_{\ell/r}}$. 
	To show that the collection $\{\sigma_{E_{\ell/r}}\}_{\ell \in \bb Z}$ defines a 
	real structure on $E_\bullet$, consider an \'etale atlas $u : U \longrightarrow \mf X$ 
	given by the quadruple $(f : U \to X, M, \phi, t) \in \mf X(U)$. 
	By definition, $\ms{M}$ is given by the line bundle $M$ on $U$. 
	For $\ell'\geq \ell$ there is a natural inclusion 
	$$j_{\ell',\, \ell} : \ms{M}^{-\ell'} \hookrightarrow \ms{M}^{-\ell}$$ 
	(see \cite[\S 3.2]{Born07}). 
	Note that $j_{\ell', \ell}$ respect the real structures 
	$\sigma_{{\ms M}^{-\ell'}}$ and $\sigma_{{\ms M}^{-\ell}}$, 
	and makes the following diagram commutative. 
	\begin{equation}
		\begin{gathered}
			\xymatrix{
			\mc{E}\otimes_{\mc{O}_{\mf X}}\ms{M}^{-\ell'} 
			\ar[rrrr]^{\Id_{\mc E}\otimes j_{\ell', \ell}} 
			\ar[dd]_{\sigma_{\mc E}\otimes\,\sigma_{{\ms M}^{-\ell'}}} 
			&&&& \mc{E}\otimes_{\mc{O}_{\mf X}}\ms{M}^{-\ell} 
			\ar[dd]^{\sigma_{\mc E}\otimes\,\sigma_{{\ms M}^{-\ell}}} \\
			&&&\\
			\left({\sigma_{\mf X}}\right)_*\left(\mc E\otimes_{\mc O_{\mf X}}
			{\ms M}^{-\ell'}\right) \ar[rrrr]^{\Id_{(\sigma_{\mf X})_*(\mc E)}
				\otimes (\sigma_{\mf X})_*(j_{\ell', \ell})} &&&& 
				\left({\sigma_{\mf X}}\right)_*\left(\mc{E}
				\otimes_{\mc{O}_{\mf X}}\ms{M}^{-\ell}\right)
		}
		\end{gathered}
	\end{equation}
	Applying $\pi_*$ to the above diagram and using 
	$\pi\circ\sigma_{\mf X} = \sigma\circ\pi$ 
	(see Proposition \ref{real structure on root stack proposition}) and the fact that 
	$E_{\frac{\ell}{r}} = \pi_*\left(\mathcal{E}\otimes_{\mc{O}_{\mf X}}
	\ms{M}^{-\ell}\right)$ we get the following commutative diagram
	\begin{equation}
		\begin{gathered}
		\xymatrix{
			E_{\frac{\ell'}{r}} \ar[r] \ar[d]_{\sigma_{({E_{\ell'/r}})}} 
			& E_{\frac{\ell}{r}} \ar[d]^{\sigma_{({E_{\ell/r}})}} \\ 
			\sigma_*\left(E_{\frac{\ell'}{r}}\right) \ar[r] & 
			\sigma_*\left(E_{\frac{\ell}{r}}\right)\,, 
		}
		\end{gathered}
	\end{equation}
	where the horizontal arrows are the are those involved in the description 
	of the functor $E_\bullet$. This shows that $E_\bullet$ is a real parabolic 
	bundle on $X$. 
	
	Conversely, suppose that $E_\bullet$ is a real parabolic vector bundle on $X$. 
	The corresponding vector bundle $\mc{E}$ on $\mf{X}$ is the coend 
	\begin{equation}
		\mathcal{E} := \int^{\frac{1}{r}\bb{Z}} \pi^*(E_\bullet)
		\otimes_{\mc{O}_{\mf X}} \ms{M}^{\bullet\,r}\,. 
	\end{equation}
	Here we remark that a coend is a special type of colimit 
	\cite[Chap. IX, \S\,5 Proposition 1]{MacL71}. 
	To show that $\mc{E}$ has an induced real structure, we first note that 
	since $\sigma_{\mf X}$ is an isomorphism, it is clear that 
	$\left(\sigma_{\mf X}\right)_*\mc{E}$ must be the coend of the direct system 
	$$\left(\sigma_{\mf X}\right)_*(\pi^*E_\bullet\otimes\ms{M}^{\bullet\,r}) 
	: \left(\frac{1}{r}\bb Z\right)^{\rm op} \times \frac{1}{r}{\bb Z} 
	\longrightarrow \Vect(\mf X)\,.$$ 
	Next, we show that the above functor is a naturally isomorphic to the following one:
	\begin{equation}
		\pi^*E_\bullet\otimes \ms{M}^{\bullet\,r} : 
		\left(\frac{1}{r}\bb Z\right)^{\rm op} \times \frac{1}{r}{\bb Z} 
		\longrightarrow \Vect(\mf X)\,. 
	\end{equation}
	Since $\pi\circ\sigma_{\mf{X}} = \sigma\circ\pi$ 
	by Proposition \ref{prop:real-str-on-root-stack}, 
	taking pullbacks 
	$\Vect(X) \overunderset{(\pi\circ\sigma_{\mf{X}})^*}{(\sigma\circ\pi)^*}{\rightrightarrows} 
	\Vect(\mf X)$ we have 
	\begin{align}
		& \sigma_{\mf{X}}^*\circ \pi^* = \pi^*\circ\sigma^* \nonumber \\
		\implies & \pi^* \simeq (\sigma_{\mf{X}})_*\circ\pi^*\circ\sigma^*,
		\,\,\,\,\,\text{since }(\sigma_{\mf{X}})_*\circ\sigma_{\mf{X}}^* 
		\simeq \Id_{\Vect(\mf{X})}. \nonumber \\
		\implies & \pi^*\circ\sigma_* \simeq (\sigma_{\mf{X}})_*\circ\pi^*, 
		\,\,\,\,\,\text{since } \,\sigma^*\circ\sigma_* \simeq \Id_{\Vect(X)}. \label{functor iso1} 
	\end{align}
	This implies that the following natural isomorphisms of functors  $\left(\frac{1}{r}\bb Z\right)^{op} \times\frac{1}{r}{\bb Z}\rightarrow \textnormal{\textbf{Vect}}(\mf X)$:
	\begin{align*}\label{isomorphism of functors of direct systems 1}
		\left(
		\sigma_{\mf{X}}\right)_*(\pi^*E_{\bullet}\otimes \ms{M}^{\bullet\,r}) 
		&\simeq \left(
		\sigma_{\mf{X}}\right)_*(\pi^*E_{\bullet})\otimes 
		\left(\sigma_{\mf{X}}\right)_*(\ms{M}^{\bullet\,r}) \\
		& \simeq \pi^*(\sigma_*E_{\bullet})\otimes 
		\left(\sigma_{\mf{X}}\right)_*(\ms{M}^{\bullet\,r})
		\,\,\,\,[\text{using}\,\,(\ref{functor iso1})] \\
		&\simeq \pi^*E_{\bullet} \otimes \ms{M}^{\bullet\,r}\,, 
	\end{align*}
	where the last isomorphism is due to the fact that both $E_\bullet$ and 
	$\ms{M}^{\bullet\,r}$ have real structures. Thus by our earlier 
	Remark \ref{rem:isomorphism-of-colimits}, we conclude that their coends, 
	namely $\mc{E}$ and $(\sigma_{\mf{X}})_*\mc{E}$ are isomorphic. 
	This gives the required real structure on $\mc{E}$. 
\end{proof}

\section{Real Connections On Root Stacks}\label{sec:root-stacks-and-real-connections}
Let $X$ be a smooth $\bb C$-variety and $D$ a reduced effective divisor on $X$ having 
{\it strict normal crossings} in the sense that all of its irreducible components 
are smooth and they intersect each others transversally. We use the abbreviation 
``sncd'' to mean a {\it strict normal crossing divisor}. 
Throughout this section we assume that $D$ is real in the sense that $\sigma(D) = D$. 

\subsection{Connections on root stacks}
We first briefly recall the notion of logarithmic connection, holomorphic 
connection and parabolic connection on bundles over schemes and root stacks. 
Let $\Omega^1_{X/\bb C}(\log D)$ be the sheaf of meromorphic differentials on $X$ 
having at most logarithmic poles along $D$. It fits into the following short exact 
sequence of $\mc O_X$-modules 
\begin{align}
	0 \rightarrow \Omega^1_{X/\bb C} \rightarrow \Omega^1_{X/\bb C}(\log D) 
	\stackrel{res}{\longrightarrow} \bigoplus\limits_{i \in I} \mc{O}_{D_i} \rightarrow 0, 
\end{align}
where $res$ denotes the residue map, $D_i$'s are components of $D$, 
and $\mc{O}_{D_i}$ denotes the structure sheaf of the closed subscheme $D_i$ of 
$X$ \cite[Properties 2.3]{EsnVieh92}. 

\begin{definition}\label{logarithmic connection definition}
	Let $E$ be a vector bundle on $X$. 
	\begin{enumerate}[(i)]
		\item A \textit{logarithmic connection} on $E$ is a $\bb C$-linear sheaf homomorphism
		\begin{align*}
			\nabla : E \to E \otimes_{\mc O_X} \Omega^1_{X/\bb C}(\log D)
		\end{align*}
		satisfying the Leibniz rule: 
		$$\nabla(f\cdot s) = f\nabla s + s\otimes df,$$
		for all locally defined sections $f$ and $s$ of $\mc O_X$ and $E$, respectively. 
		It is usually denoted as a pair $(E, \nabla)$. 
		If the image of $\nabla$ lands inside $E\otimes_{\mc O_X}\Omega^1_{X/\bb C}$, then 
		$\nabla$ is called an {\it algebraic connection} on $E$. 
		\item A \textit{morphism of logarithmic connections} from $(E,\nabla)$ to $(E',\nabla')$ 
		is a vector bundle homomorphism $$\phi: E\rightarrow E'$$
		such that the following diagram commutes: 
		\begin{equation}\label{diag:morphism-of-log-connections}
			\begin{gathered}
				\xymatrix{
					E \ar[rr]^(.3){\nabla} \ar[d]_{\phi} && 
					E\otimes \Omega^1_{X/\bb C}(\log D) \ar[d]^{\phi\otimes \Id} \\
					E' \ar[rr]^(.3){\nabla'} && E'\otimes \Omega^1_{X/\bb C}(\log D)
				}
			\end{gathered}
		\end{equation}
	\end{enumerate}
	We denote by $\Con(X, D)$ the category of logarithmic connections on $(X, D)$. 
\end{definition}
The following result gives a canonical logarithmic connection on $\mc{O}_X(D)$.

\begin{lemma}[\text{\cite[Lemma 3.7]{BornLaar23}, \cite[Lemma 2.7]{EsnVieh92}}]\label{real structure on canonical connection lemma}
	Let $B = \sum_{i\in I}\mu_i D_i$ be a Cartier divisor with support in $D$. There exists a canonical logarithmic connection:
	\begin{align}
		d(B) : \mc{O}_X(B) \rightarrow \mc{O}_X(B) \otimes_{\mc{O}_X} \Omega^1_{X/\bb C}(\log D)
	\end{align}
	characterized by $d(B)\left(\prod_{i\in I}x_i^{-\mu_i}\right) = -\sum_{i\in I}x_i^{-\mu_i}\cdot \sum_{i\in I}\mu_i\dfrac{dx_i}{x_i}$, where $x_i$ is a local equation of $D_i$.
\end{lemma}
\begin{remark}\label{rem:B-twist-of-a-log-conn}
	Let $B$ be a Cartier divisor on $X$ supported on $D$. 
	Then the tensor product of two logarithmic connections $(E, \nabla_E)$ 
	and $(F, \nabla_F)$ with logarithmic poles along $B$ is again a logarithmic 
	connection $(E\otimes F, \nabla_E\otimes\Id_F+\Id_E\otimes\nabla_F)$ with 
	logarithmic poles along $B$ (c.f. \cite{EsnVieh92} and \cite[\S\,3.3.2]{BornLaar23}). 
	Then we may define the {\it $B$-twist} of a logarithmic connection $(E, \nabla)$ 
	on $(X, D)$ to be the tensor product $(E, \nabla)\otimes(\mc O_X(B), d(B))$, 
	which is a logarithmic connection on $E(B) := E\otimes\mc O_X(B)$ having logarithmic 
	poles along $B$, where $d(B)$ is the canonical logarithmic connection on $\mc O_X(B)$ 
	as described in Lemma \ref{real structure on canonical connection lemma}. 
\end{remark}

\begin{definition}\label{parabolic connection definition}
	\begin{enumerate}[(i)]
		\item A \textit{parabolic connection} on $(X,D)$ having rational weights in $\frac{1}{r}\bb Z$ 
		along $D$ is a contravariant functor
		$$(E_\bullet\,,\,\nabla_\bullet) : \left(\frac{1}{r}\bb Z\right)^{op} 
		\longrightarrow \Con(X,D)$$ 
		such that for each positive integer $n$ there is a natural isomorphism of functors 
		$$\mf j : \left(E_\bullet\otimes {\mc O}_X(-D), 
		\nabla_\bullet(-D)\right) \longrightarrow 
		\left(E_{\bullet+1}, \nabla_{\bullet+1}\right)$$
		such that the following diagram commutes: 
		\begin{equation*}
			\xymatrix{
				\left(E_\bullet\otimes {\mc O}_X(-D), \nabla_\bullet(-D)\right) 
				\ar[rr]^-{\mf j} \ar[dr]_-{\Id_{E_\bullet}\otimes\iota} && \left(E_{\bullet+1}, \nabla_{\bullet+1}\right) 
				\ar[dl]^-{} \\ 
				& (E_\bullet, \nabla_\bullet) & 
			}
		\end{equation*}
		where $\iota$ is the canonical inclusion $\mc O_X(-D) \hookrightarrow \mc{O}_X$ 
		and $\nabla_{\bullet}(-D)$ is 
		the twist of $\nabla_{\bullet}$ by the connection $d(-D)$ of Lemma \ref{real structure on canonical connection lemma}.
		
		\item \cite[\S\,4.1]{BisMajWan12} 
		A \textit{strongly parabolic connection} on $(X, D)$ with weights in $\frac{1}{r}\bb Z$ 
		is a parabolic connection $(E_{\bullet}, \nabla_\bullet)$ on $(X, D)$ such that if 
		$\{D_i\}_{i\in I}$ are irreducible components of $D$, then for each $i\in I$ and $\ell \in \bb Z$, 
		the induced homomorphism 
		\begin{equation*}
			\overline{res_{D_i}(\nabla_{\frac{\ell}{r}})} : 
			\dfrac{E_{\frac{\ell}{r}}\big\vert_{D_i}}{E_{\frac{\ell+1}{r}}\big\vert_{D_i}} 
			\longrightarrow \dfrac{E_{\frac{\ell}{r}}\big\vert_{D_i}}{E_{\frac{\ell+1}{r}}\big\vert_{D_i}}
		\end{equation*} 
		is equal to $\frac{\ell}{r}\cdot\Id$, where 
		$res_{D_i}(\nabla_{\frac{\ell}{r}}) : E_{\frac{\ell}{r}}\big\vert_{D_i} 
		\longrightarrow E_{\frac{\ell}{r}}\big\vert_{D_i}$ is the residue map 
		(c.f. \cite[Definition 2.5]{EsnVieh92}).
	\end{enumerate}
\end{definition}

\begin{remark}\label{push-forward of a logarithmic connection remark}
	Let $(E, \nabla)$ be a logarithmic connection on $X$ with poles along $D$. 
	Since $\sigma$ is an isomorphism and $\sigma^*(D) = D$, 
	it follows that the direct image along $\sigma$ of a logarithmic connection is 
	well-defined, and has poles along $D$ (see \cite[Remark 3.8]{BornLaar23}). 
	Thus we get a logarithmic connection $\sigma_*\nabla$ on $\sigma_*E$: 
	$$\sigma_*\nabla : \sigma_*E \longrightarrow 
	\sigma_*E\otimes_{\mc{O}_X}\Omega^1_{X/\bb C}(\log D).$$
\end{remark}

Now we define the notions of {\it real logarithmic connections} and {\it real parabolic connections}. 

\begin{definition}\label{def:real-logarithmic-connection}
	Let $D$ be a real sncd on a real variety $(X, \sigma)$, and let 
	$(E, \sigma_E)$ be a real vector bundle on $(X, \sigma)$. 
	\begin{enumerate}[(i)]
		\item  A logarithmic connection $\nabla$ on $E$ is said to be a 
		\textit{real logarithmic connection} if $\sigma_E$ is a morphism of connections 
		between $(E, \nabla)$ and $(\sigma_{*}E, \sigma_*\nabla)$ in the sense of 
		Definition \ref{logarithmic connection definition}; 
		in other words, if the diagram below commutes:
		\begin{equation}\label{diag:def-of-real-log-connection}
			\begin{gathered}
				\xymatrix{
					E \ar[rr]^-{\nabla} \ar[d]_{\sigma_E} && 
					E\otimes\Omega^1_{X/\bb C}(\log D) \ar[d]^{\sigma_E\otimes\Id} \\
					\sigma_*E \ar[rr]^-{\sigma_*\nabla} && 
					\sigma_*E\otimes\Omega^1_{X/\bb C}(\log D)
				}
			\end{gathered}
		\end{equation}
		A \textit{real algebraic 
		connection} is defined similarly.

		\item A \textit{real parabolic connection} on $(X, \sigma)$ is a parabolic connection 
		$$(E_{\bullet}, \nabla_{\bullet}) : \left(\frac{1}{r}\bb Z\right)^{\rm op} 
		\longrightarrow \textnormal{Con}(X, D)$$ 
		together with a natural isomorphism of functors 
		$\sigma_{E_\bullet} : E_{\bullet} \longrightarrow \sigma_*(E_{\bullet})$ 
		such that for each $\ell \in \bb Z$ the following diagram is commutative: 
		\begin{equation*}
			\begin{gathered}
				\xymatrix{
					E_{\frac{\ell}{r}} \ar[rr]^-{\nabla_{\frac{\ell}{r}}} 
					\ar[d]_{\sigma_{E_{\ell/r}}} 
					&& E_{\frac{\ell}{r}} \otimes \Omega^1_{X/\bb C}(\log D) 
					\ar[d]^{\sigma_{E_{\ell/r}}\otimes\Id} \\
					\sigma_*(E_{\frac{\ell}{r}}) \ar[rr]^-{\sigma_*\left(\nabla_{\frac{\ell}{r}}\right)} 
					&& \sigma_*(E_{\frac{\ell}{r}}) \otimes\Omega^1_{X/\bb C}(\log D)
				}
			\end{gathered}
		\end{equation*}
	
	\end{enumerate}

\end{definition}


\begin{lemma}\label{lem:canonical-real-log-connection-on-O_X(D)}
	Let $B = \sum_{i\in I}\mu_i D_i$ be a Cartier divisor with support in $D$. 
	Assume that $D_j$ is locally defined by $x_j$, for all $j \in I$. 
	Then the canonical logarithmic connection 
		$$d(B) : \mc{O}_X(B) \rightarrow \mc{O}_X(B)\otimes\Omega^1_{X/\bb C}(\log D) $$
		of Lemma \ref{real structure on canonical connection lemma} is real.
\end{lemma}

\begin{proof}
	Since ${\rm Supp}(B) = D$ is real, so is $B$, i.e., $\sigma^*B = B$. 
	Then the invertible sheaf $\mc O_X(B)$ admits a real structure $\sigma_{\mc O_X(B)}$, 
	(see Proposition \ref{prop:real-str-on-O_X(D)}). 
	We claim that $d(B)$ is a real logarithmic connection with respect to $\sigma_{\mc O_X(B)}$. 
	The condition $\sigma^*(B) = B$ ensures that the local equations for $B$ and $\sigma_*B$ 
	remain the same, and so $\sigma_{\mc O_X(B)}$ simply permutes the local equations for $B$. 
	Now it follows from the description of $d(B)$ on local sections that the 
	diagram \eqref{diag:def-of-real-log-connection} commutes for $(E, \nabla) = (\mc O_X(B), d(B))$. 
\end{proof}

\begin{proposition}\label{prop:tensor-product-of-real-log-connections}
	As before, let $B = \sum\limits_{i \in I} \mu_j D_j$ be a divisor on $X$ 
	supported in $D$. If $(E, \nabla_E)$ and $(F, \nabla_F)$ are real logarithmic connections 
	on $(X, \sigma)$ with logarithmic poles along $B$, so are their direct sum and tensor product. 
	In particular, $B$-twist of any real logarithmic connection is a real logarithmic connection. 
\end{proposition}

All the definitions and results in this section have their counterparts 
for the corresponding root stacks. An {\it effective divisor} on $\mf X$ is a 
closed substack $\mf D \subset \mf X$ of $\mf X$ such that its ideal sheaf 
$\mc I_{\mf D} \subseteq \mc O_{\mf X}$ is invertible (equivalently, 
for any \'etale atlas $U \to \mf X$ of $\mf X$, 
$\mf D\times_{\mf{X}} U$ is an effective Cartier divisor on $U$). 
An effective Cartier divisor $\mf D$ on $\mf X$ is said to have 
{\it normal crossings} if this is true in an \'etale chart of $\mf X$. 
Moreover, if $\mf D$ has normal crossing and all of its irreducible components 
are regular, we say that $\mf D$ is a {\it strict normal crossing divisor} on $\mf X$ 
(see \cite[pp.~9]{BornLaar23}). 
Let $D$ be an effective sncd on $X$. 
Then the associated $r$-th root stack $\mf X = \mf{X}_{(\mc{O}_X(D),\, s_D,\, r)}$ 
is smooth \cite[Proposition 2.4]{BornLaar23}. 
Let $\pi : \mf X \longrightarrow X$ be the canonical morphism. 
The tautological invertible sheaf $\ms M$ on $\mf X$ gives rise to an 
effective divisor $\mf D$ on $\mf X$, which is also a strict normal crossing 
divisor \cite[\S 3.3.1]{BornLaar23} satisfying the relation $\pi^*(D) = r\mf D$. 
Moreover, as a log-pair, the map $(\pi, r) : (\mf X, \mf D) \longrightarrow (X, D)$ 
is a log-\'etale morphism \cite[\S 3.3.1]{BornLaar23}. 
In particular, for a sncd $\mf B$ on $\mf X$ \cite[Definition 2.6]{BornLaar23}, 
by considering the \'etale atlases on $\mf X$, one obtains a stacky version of 
the sheaf of logarithmic differentials for stacks, 
namely $\Omega^1_{{\mf X}/{\bb C}}(\log\mf B)$ on $\mf X$ [\S\,3.3.1, \textit{loc. cit.}], 
and a canonical logarithmic connection $d(\mf B)$ on the line bundle $\mc{O}_{\mf X}(\mf B)$.

\begin{lemma}\label{lem:logarithmic-differential-isomorphism}
	The induced morphism 
	$\pi^*\Omega^1_{X/\bb C}(\log(D)) \longrightarrow \Omega^1_{{\mf X}/{\bb C}}(\log(\mf D))$ 
	under the log-\'etale morphism $\pi$ is an isomorphism. 
\end{lemma}

\begin{proof}
	This is a stacky version of \cite[Lemma 3.5]{BornLaar23}; 
	a similar proof works by considering \'etale atlases for $\mf X$. 
\end{proof}

We refer the reader to \cite[\S 2.1.3, \S 3.3.2]{BornLaar23} for more details.


\subsection{Real logarithmic connection and real parabolic connection}
It is shown in \cite{BisMajWan12} that there is an equivalence of the category 
of logarithmic connections on $(\mf X, \mf D)$ and the category of parabolic connections 
on $(X, D)$ having parabolic weights in $\frac{1}{r}\bb{Z}$ along $D$. 
Furthermore, this equivalence restricts to give an equivalence of the category of 
holomorphic connections on $\mf X$ with the category of strongly parabolic connections 
on $(X, D)$ with parabolic weights in $\frac{1}{r}\bb Z$ along $D$. It is given as follows.

\begin{enumerate}[(1)]
	\item \cite[Definition 4.18]{BornLaar23} 
	To each logarithmic connection $(\mc E, \nabla)$ on $(\mf X, \mf D)$, one associates 
	a parabolic connection $(E_{\bullet},\widehat{\nabla}_{\bullet})$ on $(X,D)$, 
	where 
	\begin{enumerate}[(i)]
		\item $E_{\bullet}$ is the parabolic bundle given as in 
		Theorem \ref{thm:Biswas-Borne-correspondence}, and 
		
		\item $\widehat{\nabla}_{\bullet}$ is given by 
		$$\widehat{\nabla}_{\frac{\ell}{r}} = \pi_*(\nabla(-\ell\mf D))\,,$$
		where $\nabla(-\ell\mf D)$ is the tensor product of the connections 
		$\nabla$ an $d(-\ell\mf D)$. 
		(Note that, the direct image of connection makes sense because the 
		morphism of log-pairs $(\pi, r) : (\mf X, \mf D) \longrightarrow (X, D)$ 
		is log-\'etale, see \cite[Remark 3.8 and Remark 4.19]{BornLaar23}). 
	\end{enumerate}
	
	\item \cite[Definition 4.29]{BornLaar23} 
	On the other hand, for any such parabolic connection 
	$(E_{\bullet}, \widehat{\nabla}_{\bullet})$ on $(X, D)$, one associates the 
	logarithmic connection $(\mc E, \nabla)$ on $\mf X$, where 
	$$\mc E = \int^{\frac{1}{r}\bb{Z}} 
	\pi^*(E_{\bullet})\otimes_{\mc O_{\mf X}} \mc O_{\mf X}(\bullet\,r \mf D)$$ 
	stands for the coend as in Theorem \ref{thm:Biswas-Borne-correspondence}, 
	and $\nabla$ is the unique connection on the coend $\mc E$ compatible with 
	the given connections at each term of the coend (the existence of such a 
	connection is guaranteed by \cite[Lemma 4.27]{BornLaar23}). 
\end{enumerate}



\begin{theorem}\label{thm:real-log-parbolic-connection-correspondence}
	Let $(X, \sigma)$ be a real variety and $D$ a reduced effective sncd satisfying 
	$\sigma^*(D) = D$. Let $\mf X$ be the $r$-th root stack associated to the divisor $D$. 
	Then there is an equivalence between the category of real logarithmic connections on 
	$(\mf X, \mf D)$ and the category of real parabolic connections on $(X, D)$ having 
	weights in $\frac{1}{r}\bb Z$. 
	
	Moreover, under this correspondence, the category of real holomorphic connections on 
	$\mf X$ is equivalent to the category of real strongly parabolic connections on $(X,D)$ 
	having weights in $\frac{1}{r}\bb Z$. 
\end{theorem}
\begin{proof}
	It is enough to prove that the correspondence stated above preserves the real structures 
	in question. Let $\nabla$ be a real logarithmic connection on a real vector bundle 
	$(\mc E, \sigma_{\mc E})$ on $(\mf X, \mf D)$. For each integer $\ell$, we can take 
	the tensor product of this connection with the canonical connection $d(-\ell\mf D)$ on 
	$\mc{O}_{\mf X}(-\ell\mf D)$, which leads us to the following commutative diagram: 
	\begin{equation*}
		\xymatrix{
			\mc{E}(-\ell\mf D) \ar[rrr]^-{\nabla(-\ell\mf D)} 
			\ar[d]_{\sigma_{\mc E(-\ell\mf D)}} &&& 
			\mc{E}(-\ell\mf D)\otimes_{\mc{O}_{\mf{X}}} \Omega^1_{\mf X/\bb C}(\log \mf D) 
			\ar[d]^{\sigma_{\mc E(-\ell\mf D)}\otimes\Id} \\ 
			(\sigma_{\mf X})_*{(\mc E(-\ell\mf D))} 
			\ar[rrr]^-{(\sigma_{\mf X})_*\nabla(-\ell\mf D)} 
			&&& (\sigma_{\mf X})_*{(\mc E(-\ell\mf D))} 
			\otimes_{\mc{O}_{\mf{X}}} \Omega^1_{\mf{X}/\bb C}(\log \mf D)
		}
	\end{equation*}
	where $(\sigma_{\mf X})_*\nabla(-\ell\mf D)$ is taken in the same sense as in 
	Definition \ref{def:real-logarithmic-connection} for its stacky counterpart. 
	Applying $\pi_*$ on this diagram and using 
	$\pi_*\circ (\sigma_{\mf{X}})_* = \sigma_*\circ\pi_*$ 
	(c.f. Proposition \ref{prop:real-str-on-root-stack}) 
	and the fact that 
	$\pi^*\Omega^1_{X/\bb C}(\log D) \cong \Omega^1_{\mf X/\bb C}(\log \mf D)$ 
	(c.f.Lemma \ref{lem:logarithmic-differential-isomorphism}) 
	together with projection formula, we get the following diagram:
	\begin{equation*}
		\xymatrix{
			E_{\frac{\ell}{r}} \ar[rrr]^-{\widehat{\nabla}_{\frac{\ell}{r}}} 
			\ar[d]_{\sigma_{E_{\ell/r}}} &&& 
			E_{\frac{\ell}{r}} \otimes_{\mc O_X} \Omega^1_{X/\bb C}(\log D) 
			\ar[d]^{\sigma_{E_{\ell/r}}\otimes\Id} \\
			\sigma_*(E_{\frac{\ell}{r}}) 
			\ar[rrr]^-{\sigma_*\left(\widehat{\nabla}_{\frac{\ell}{r}}\right)} 
			&&& \sigma_*(E_{\frac{l}{r}}) \otimes_{\mc O_X} \Omega^1_{X/\bb C}(\log D)
		}
	\end{equation*} 
	This produces a real parabolic connection $\widehat{\nabla}_{\bullet}$ 
	on the real parabolic bundle $(E_{\bullet}, \sigma_{E_\bullet})$. 
	
	Conversely, given a real parabolic connection $\widehat{\nabla}_{\bullet}$ 
	on a real parabolic bundle $(E_{\bullet}, \sigma_{E_\bullet})$, the assocaited coend 
	$$\mc{E} = \int^{\frac{1}{r}\bb{Z}} 
	\pi^*(E_{\bullet})\otimes_{\mc O_{\mf X}}\mc O_X(\bullet\,r\mf D)$$ 
	admits a real structure $\sigma_{\mc E}$ 
	(see Theorem \ref{thm:real-avatar-of-Biswas-Borne-correspondence}). 
	Moreover, $\mc E$ admits a logarithmic connection $\nabla$ compatible with the 
	given connections on each pieces of the coend, namely 
	$\pi^*(\widehat{\nabla}_{\bullet})\otimes d(\bullet\,r\mf D)$ 
	(see \cite[Lemma 4.27]{BornLaar23}). 
	Consequently, the pair $(\mc E, \nabla)$ can be interpreted as a 
	coend in the category $\Con(\mf X, \mf D)$: 
	\begin{equation}\label{def:eqn-for-coend-of-parabolic-connection}
		(\mc E, \nabla) = \int^{\frac{1}{r}\bb{Z}} 
		\left(\pi^*(E_{\bullet})\otimes_{\mc{O}_{\mf{X}}} 
		{\mc O}_{\mf X}(\bullet\,r\mf D), 
		\pi^*(\widehat{\nabla}_{\bullet})\otimes d(\bullet r\mf D)
		\right). 
	\end{equation}
	Consider the logarithmic connection $(\left(\sigma_{\mf X}\right)_*\nabla)$ on 
	$\left(\sigma_{\mf X}\right)_*\mc E$. Since $\sigma_{\mf X}$ is an isomorphism, 
	it follows that $\left((\sigma_{\mf X})_*\mc E, (\sigma_{\mf X})_*\nabla\right)$ 
	is the coend of the direct system 
	\begin{align}\label{coend of connections 2}
		\left(\left(
		\sigma_{\mf{X}}\right)_*\left(\pi^*E_{\bullet}\otimes_{\mc{O}_{\mf X}} \mc{O}_{\mf X}(\bullet r\mf D)\right), \left(\sigma_{\mf{X}}\right)_*((\pi^*\widehat{\nabla}_{\bullet})\otimes d(\bullet r\mf D))\right): \left(\frac{1}{r}\bb Z\right)^{op}\times\frac{1}{r}\bb Z\rightarrow \textnormal{Con}(\mf X,\mf D) 
	\end{align}
	(Alternately, one can use the uniqueness part of \cite[Lemma 4.27]{BornLaar23} 
	to conclude the same.) 

	Next, we have seen in the proof of Theorem \ref{thm:real-avatar-of-Biswas-Borne-correspondence} that there exists an isomorphism of functors between the two direct systems of underlying bundles 
	$$\left(
	\sigma_{\mf{X}}\right)_*\left(\pi^*E_{\bullet}\otimes_{\mc{O}_{\mf X}} \mc{O}_{\mf X}(\bullet r\mf D)\right)\,\, \text{and}\,\, \pi^*E_{\bullet}\otimes_{\mc{O}_{\mf X}} \mc{O}_{\mf X}(\bullet r\mf D) : \left(\frac{1}{r}\bb Z\right)^{op} \times\frac{1}{r}{\bb Z}\rightarrow \textnormal{\textbf{Vect}}(\mf X)$$
	We claim that this is in fact an isomorphism of functors between the direct systems of connections in \eqref{def:eqn-for-coend-of-parabolic-connection} and \eqref{coend of connections 2}, in other words, we need to check its compatibility with their respective logarithmic connections $\left(\sigma_{\mf{X}}\right)_*((\pi^*\widehat{\nabla}_{\bullet})\otimes d(\bullet r\mf D))$ and $(\pi^*\widehat{\nabla}_{\bullet})\otimes d(\bullet r\mf D)$.
	\begin{enumerate}[(1)]
		\item We can write $\left(\sigma_{\mf X}\right)_*(\pi^*E_{\bullet})\simeq \sigma_*(\pi^*E_{\bullet})$, and it follows from \eqref{isomorphism of functors of direct systems 1} that the functorial isomorphism in question between 
		$$\pi^*E_{\bullet}\otimes_{\mc{O}_{\mf X}} \mc{O}_{\mf X}(\bullet r\mf D)\,\,\text{and}\,\, \left(
		\sigma_{\mf{X}}\right)_*\left(\pi^*E_{\bullet}\,\underset{{\mc{O}_{\mf{X}}}}{\otimes}\, \mc{O}_{\mf X}(\bullet r\mf D)\right)\simeq \pi^*(\sigma_*(E_{\bullet}))\,\underset{{\mc{O}_{\mf{X}}}}{\otimes}\,\left(\sigma_{\mf{X}}\right)_*\mc{O}(\bullet r\mf D)$$
		is given precisely by $\pi^*\left(\sigma_{E_\bullet}\right)\otimes \sigma_{\mc{O}(\bullet r\mf D)}$\ ,\\
		\item Since $\widehat{\nabla}_{\bullet}$ is compatible with the real structure $\sigma_{E_\bullet}$, its pull-back $\pi^*(\widehat{\nabla}_{\bullet})$ is compatible with the real structure $\pi^*(\sigma_{E_{\bullet}})$ on $\pi^*(E_{\bullet})$. On the other hand, the isomorphism $\mc{O}_{\mf X}(\bullet r\mf D) \simeq \pi^*(\mc{O}_X(\bullet D))$ respects the real structures $\sigma_(\bullet  r\mf D)$ and $\pi^*(\sigma_{\mc{O}(\bullet D)})$ [Lemma \ref{real structure preserving isomorphism}], and since the connections $d(\bullet D)$ are compatible with the real structures $\sigma_{\mc{O}(\bullet D)}$, it follows that their pull-backs $\pi^*(d(\bullet D))$ are compatible with $\pi^*(\sigma_{\mc{O}(\bullet D)})$, which implies that the connections $d(\bullet r\mf D)$ are compatible with $\sigma_{\mc{O}(\bullet r\mf D)}$.\\
		\item Thus, the tensor product connection $\pi^*(\widehat{\nabla}_{\bullet})\otimes d(\bullet r\mf D)$ is compatible with the tensor product of the real structures $\pi^*(\widetilde{\sigma}^{(E_{\boldsymbol{\cdot}})})\otimes \sigma_{\mc{O}(\bullet r\mf D)}$\ , which is precisely the functorial isomorphism in $(1)$.
	\end{enumerate}
	Thus the direct systems of connections in \eqref{def:eqn-for-coend-of-parabolic-connection} and \eqref{coend of connections 2} are isomorphic, which implies that their colimits, namely the coends $(\mc E, \nabla)$ and $\left(\left(\sigma_{\mf X}\right)_*\mc E, \left(\sigma_{\mf X}\right)_*\nabla\right)$ are isomorphic, by our Remark \ref{rem:isomorphism-of-colimits}. This isomorphism makes $\nabla$ into a real logarithmic connection on the real vector bundle $(\mc E, \sigma_{\mc E})$.
\end{proof}

\section{Root stacks and ramified Galois covers}
In this section, we wish to take a different approach to prove some of the results in the previous section for a smooth projective variety $X$ from the perspective of group actions and ramified covers. \\
To facilitate our discussion, let us first consider the case when a real variety has an action of a finite group compatible with the real structure in a certain sense, and prove some results related to this situation. These results will be eventually used in our main theorem in this section.\\
Let $Y$ be a complex variety with a real structure $\sigma_Y$. Let $\Gamma$ be a finite subgroup of $ \text{Aut}(Y)$, the group of algebraic automorphisms of $Y$. Thus $\Gamma$ acts faithfully on $Y$. Moreover, let us assume that
\begin{align}\label{group action}
	\sigma_Y(g\cdot y) = g^{-1} \sigma_Y(y) \,\,\forall\,y\in Y\,.
\end{align}
A $\Gamma$-equivariant vector bundle on $Y$ is a vector bundle $E$ on $Y$ with a compatible $\Gamma$-action (see, e.g. \cite[\S1.3]{BisMajWan12}). It is not hard to see that this induces an isomorphism of vector bundles 
\begin{align}\label{equivariant bundle conditions 1}
	\rho_g : E \simeq g^*E
\end{align}
for every $g\in \Gamma$.\\

\begin{definition}\label{real equivariant connection definition}
	Let $(E,\widetilde{\sigma}_Y^E)$ be a real $\Gamma$-equivariant vector bundle on $(Y,\sigma_Y)$ (see \cite[Definition 4.4]{Amrutiya14a} for the definition of real equivariant bundles). A \textit{real} $\Gamma$-\textit{holomorphic connection} on $E$ is a holomorphic connection $$\nabla: E\rightarrow E\otimes\Omega^1_Y$$
	which satisfies the following two properties:
	\begin{enumerate}[(a)]
		\item For every $g\in \Gamma$, the following diagram commutes:
		\begin{align}\label{real connection diagram 1}
			\xymatrix@=1.3cm{
				E \ar[d]_{\rho_g} \ar[r]^(.4){\nabla} & E\otimes\Omega^1_Y 
				\ar[d]^{\rho_{g}\otimes\text{Id}} \\
				g^*E \ar[r]^(.4){g^*\nabla} & g^*E\otimes \Omega^1_Y
			}
		\end{align}
		where $g^*\nabla$ is the unique pull-back connection on $g^*E$\,(in other words, the isomorphism $\rho_g$  takes $\nabla$ to the pull-back connection $g^*\nabla$).\\
		\item $\nabla$ is compatible with the real structure making the diagram below commute:
		\begin{align}\label{real connection diagram 2}
			\xymatrix{
				E \ar[rr]^{\nabla} \ar[dd]_{\widetilde{\sigma}_Y^{E}} && E\otimes\Omega^1_Y \ar[dd]^{\widetilde{\sigma}_Y^{E}\otimes \textnormal{Id}} \\ 
				&&\\
				(\sigma_Y)_*E \ar[rr]^(.4){(\sigma_Y)_*\nabla} && (\sigma_Y)_*E \otimes \Omega^1_Y
			}
		\end{align}
		where $(\sigma_Y)_*\nabla$ is is as in Remark \ref{push-forward of a logarithmic connection remark}.
	\end{enumerate}
\end{definition}

\subsection{Real structure on quotient stack}\hfill\\
Let us now work with a real variety $(Y,\sigma_Y)$ together with an action of a finite group $\Gamma$ satisfying \eqref{group action}. 
\begin{proposition}\label{real structure on quotient stack proposition}
	The quotient stack $[Y/\Gamma]$ has a real structure induced from the real structure on $Y$.
\end{proposition}

\begin{proof}
	We need to produce a morphism $$\sigma_{[Y/\Gamma]}: [Y/\Gamma]\rightarrow [Y/\Gamma]$$
	with the property that $\sigma_{[Y/\Gamma]}\circ \sigma_{[Y/\Gamma]}$ is $2$-isomorphic to the identity morphism. For any scheme $T$, we define the map
	\begin{equation*}
		\sigma_{[Y/\Gamma]}(T) : [Y/\Gamma](T) \rightarrow [Y/\Gamma](T)
	\end{equation*}
	in the following manner. An object $\tau$ of $[Y/\Gamma](T)$ corresponds to the data 
	\begin{align*}
		\xymatrix{
			P \ar[r]^{h} \ar[d]_{p} & Y \\
			T & 
		}
	\end{align*}
	where $p$ makes $P$ into a $\Gamma$-torsor over $T$, and $h$ is $\Gamma$-equivariant. Let us consider the fiber product
	\begin{align*}
		\xymatrix{
			\sigma_Y^*P \ar[r]^{h'} \ar[d]_{l} & Y \ar[d]^{\sigma_Y} \\
			P \ar[r]^{h} & Y
		}
	\end{align*}
	where 
	\begin{align*}
		\sigma_Y^*P & \simeq \{(v,y)\in P\times Y\,\mid\, h(v) = \sigma_Y(y)\}\\
		& = \{(v,y)\,\mid\,y=(\sigma_Y\circ h)(v)\}\\
		& =\{(v,\left(\sigma_Y\circ h)(v)\right)\,\mid\,v\in P\}
	\end{align*}
	and under this identification, the maps $h'$ and $l$ correspond to the projections onto $Y$ and $P$ respectively.
	From the fact that $(\sigma_Y\circ h)(g^{-1}v)=\sigma_Y(g^{-1}\cdot h(v))= g\cdot(\sigma_Y\circ h)(v)$, it is clear from the description of $\sigma_Y^*P$ that the $\Gamma$-action on $P\times Y$ given by 
	$$g\cdot(v,y)= (g^{-1}v, gy)\,,$$
	makes $\sigma_Y^*P$ a $\Gamma$-invariant subset of $P\times Y$, and under this action, the map $h'$ (which corresponds to the second projection) is $\Gamma$-equivariant.
	We then define the object $\sigma_{[Y/\Gamma]}(T)(\tau)$ as the following data:
	\begin{align}
		\xymatrix{ \sigma_Y^*P \ar[r]^{h'} \ar[d]_{p\circ l} & Y \\
			T & 
		}
	\end{align}
	It is now an easy matter to check that this gives us a real structure $\sigma_{[Y/\Gamma]}$ on $[Y/\Gamma]$. 
\end{proof}

\begin{theorem}\label{real connection on quotient stack correspond to equivariant connections theorem}
	Let $(Y,\sigma_Y)$ a real variety with an action of a finite group $\Gamma$ compatible with $\sigma_Y$ as above. The category of real holomorphic connections on the quotient stack $[Y/\Gamma]$ with the real structure $\sigma_{[Y/\Gamma]}$ is equivalent to the category of real $\Gamma$-holomorphic connections on $(Y,\sigma_Y)$.
\end{theorem}

\begin{proof}
	Suppose $\widehat{\nabla}$ be a real holomorphic connection on a real vector bundle $(\mc{E},\widetilde{\sigma}_{[Y/\Gamma]}^{\mc{E}})$ on $[Y/\Gamma]$. The $\Gamma$-torsor $q: Y \rightarrow [Y/\Gamma]$ is an \'etale atlas. We have the cartesian diagram:
	\begin{align*}
		\xymatrix{
			\Gamma\times Y \ar[r]^{\lambda} \ar[d]_{p_{Y}} & Y \ar[d]^{q} \\
			Y \ar[r]^{q} & [Y/\Gamma]
		}
	\end{align*}
	where $\lambda$ and $p_Y$ denote the $\Gamma$-action and the second projection respectively. Then clearly $q\circ p_Y:\Gamma\times Y \rightarrow [Y/\Gamma]$ is also an \'etale atlas for $[Y/\Gamma]$, and consequently, for the associated vector bundle $E := \mc{E}_q$ on $Y$ together with a connection $\nabla := \widehat{\nabla}_q$\,, we have an isomorphism 
	\begin{align*}
		\mu_1: \lambda^*E \xrightarrow[]{\simeq} \mc{E}_{q\circ p_Y} 
	\end{align*}
	similarly, there is an isomorphism
	\begin{align*}
		\mu_2: p_Y^*E \xrightarrow[]{\simeq} \mc{E}_{q\circ \lambda}\,.
	\end{align*}
	Since $\mc{E}_{q\circ p_Y} \simeq \mc{E}_{q\circ\lambda}$\,, we get an isomorphism
	\begin{align}\label{isomorphism 1}
		\mu: \lambda^*E \xrightarrow[]{\simeq} p_Y^*E
	\end{align} 
	
	Now, from the definition of a connection in terms of atlases we get the following two commutative diagrams, namely
	\begin{align}\label{diagram 1}
		\xymatrix@=1.7cm{
			\lambda^*E \ar[d]_{\mu_1} \ar[r]^(.4){\lambda^*\nabla} &  \lambda^*E\otimes \Omega^1_{\Gamma\times Y} \ar[d]^{\mu_1\otimes\text{Id}} \\
			\mc{E}_{q\circ p_Y} \ar[r]^(.4){\widehat{\nabla}_{q\circ p_{Y}}} & \mc{E}_{q\circ p_Y} \otimes \Omega^1_{\Gamma\times Y}
		}
	\end{align}
	
	and 
	\begin{align}\label{diagram 2}
		\xymatrix@=1.7cm{
			p_Y^*E \ar[d]_{\mu_2} \ar[r]^(.4){p_Y^*\nabla} &  p_Y^*E\otimes \Omega^1_{\Gamma\times Y} \ar[d]^{\mu_2\otimes\text{Id}} \\
			\mc{E}_{q\circ \lambda} \ar[r]^(.4){\widehat{\nabla}_{q\circ \lambda}} & \mc{E}_{q\circ \lambda} \otimes \Omega^1_{\Gamma\times Y}
		}
	\end{align}
	Since the isomorphism $\mc{E}_{q\circ p_Y} \simeq \mc{E}_{q\circ\lambda}$ is compatible with the connections, it now easily follows that we can combine 
	the diagrams \eqref{diagram 1} and \eqref{diagram 2} to get the following diagram:
	\begin{align}\label{diagram 3}
		\xymatrix@=1.7cm{
			\lambda^*E \ar[d]_{\mu} \ar[r]^(.4){\lambda^*\nabla} & \lambda^*E\otimes \Omega^1_{\Gamma\times Y} \ar[d]^{\mu\otimes\text{Id}} \\
			p_Y^*E \ar[r]^(.4){p_Y^*\nabla} & p_Y^*E \otimes \Omega^1_{\Gamma\times Y}
		}
	\end{align}
	For any $g\in\Gamma$, let $\iota_g: \{g\}\times Y\hookrightarrow \Gamma\times Y$ denote the inclusion. Clearly,
	$$\iota_g^*(\lambda^*E)\simeq g^*E\,\,\,\textnormal{and}\,\,\,\iota_g^*(p_Y^*E)\simeq E$$
	under the natural idenfication $Y\simeq \{g\}\times Y$\,. Thus,
	Pulling back this diagram via $\iota_g$ and identifying $Y\simeq \{g\}\times Y$ gives us the commutative diagram
	\begin{align*}
		\xymatrix@=1.7cm{
			g^*E\ar[r]^{\rho_g} \ar[d]_{g^*\nabla} & E \ar[d]^{\nabla} \\
			g^*E\otimes \Omega^1_Y \ar[r]^{\rho_g\otimes\text{Id}} & E\otimes \Omega^1_Y
		}
	\end{align*}
	where $\rho_g$ is as in \eqref{equivariant bundle conditions 1}. But this is precisely the commutativity condition in \eqref{real connection diagram 1}. Thus $(\mc{E},\nabla)$ is a $\Gamma$-holomorphic connection on $Y$. \\
	To see that $\nabla$ is compatible with the real structure $\sigma_Y$ in the sense of Definition \ref{real equivariant connection definition} $(b)$, let us first mention that the following diagram is Cartesian, which is easy to see from the fact that $\sigma_Y$ is an isomorphism:
	\begin{align*}
		\xymatrix@=1.5cm{
			Y \ar[r]^{\sigma_Y} \ar[d]_{q} & Y \ar[d]^{q}\\
			[Y/\Gamma] \ar[r]^{\sigma_{[Y/\Gamma]}} & [Y/\Gamma]
		}
	\end{align*}
	Thus, following the description of direct images in terms of atlases as in 
	\S\,\ref{sec:real-str-on-the-tautological-line-bundle} 
	the diagram \eqref{real holomorphic connection for stack diagram} in this case amounts to the following diagram on the \'etale atlas $q:Y\rightarrow [Y/\Gamma]$:
	\begin{align*}
		\xymatrix@=2cm{
			E \ar[d]_{\widetilde{\sigma}_Y^{E}} \ar[r]^{\nabla} & E \otimes_{\mc{O}_{Y}} \Omega^1_{Y} \ar[d]^{\widetilde{\sigma}_Y^{E}\otimes\textnormal{Id}}\\
			(\sigma_{Y})_*E \ar[r]^(0.4){(\sigma_{Y})_*\nabla} & (\sigma_Y)_*E\otimes_{\mc{O}_Y}\Omega^1_{Y}
		}
	\end{align*}
	with the notations as in the diagram \eqref{real connection diagram 2}.
	Thus $\nabla$ respects the real structure $\sigma_Y$, which enables us to conclude that $\nabla$ is a real $\Gamma$-holomorphic connection on $(Y,\sigma_Y)$. 
\end{proof}

\subsection{Finite Galois coverings}\label{the case of curves subsection}\hfill\\
Let $X$ be an irreducible smooth complex projective variety together with a real structure $\sigma$. Fix a sncd $D$ on $X$ satisfying $\sigma^*D = D$. Let $(D_i)_{1\leq i\leq s}$ be the irreducible components of $D$.
Fix an integer $r>1$. By a covering lemma of Kawamata \cite[Lemma 5]{Kaw82}, there is a 
finite Galois covering $p : Y \to X$ tamely ramified along $D$ with ramification 
multi-index ${\bf r} = (r_1, \ldots, r_s)$, where each $r_i$ is a multiple of $r$. Such a covering $p : Y \to X$ is 
constructed as a sequence of finite surjective morphisms 
\begin{equation}
	Y := Y_N \stackrel{p_N}{\longrightarrow} Y_{N-1} \stackrel{p_{N-1}}{\longrightarrow} \cdots \stackrel{p_2}{\longrightarrow} Y_1 \stackrel{p_1}{\longrightarrow} Y_0 := X, 
\end{equation}
where $p_{i+1} : Y_{i+1} \to Y_i$ is a normalization of $Y_i$ in some field 
extension $L_i$ of the function field $K(Y_i)$, for all $i \geq 0$ 
(see proof of \cite[Theorem 17]{Kaw81}). 
Then applying universal property of normalization \cite[Proposition 12.44]{GorWed20} to the real structure $\sigma_{Y_i}$ for $Y_i$, we get a real structure $\sigma_{Y_{i+1}}$ on $Y_{i+1}$ compatible with the real structure on $Y_i$ along the normalization map $p_{i+1}$, for $i = 0, 1, \ldots, N-1$. Thus we get a real structure $\sigma_Y$ on $Y$ compatible with the real structure on $X$. 
If $\Gamma$ denotes the Galois 
group of the cover $p$, then $\Gamma$ satisfies 
$$\sigma_Y(gy) = g^{-1}\sigma_Y(y)\,\,\,\forall\,g\in\Gamma,y\in Y\,.$$


\begin{proposition}\label{root stack isomorphic to quotient stack proposition}
	Suppose that the ramified Galois cover $p:Y\rightarrow X$ constructed above has the following additional properties:
	\begin{enumerate}[$\bullet$]
		\item $p$ is ramified precisely along $D$, and 
		\item for each point of the ramification locus $\widetilde{D} := (p^*D)_{\text{red}}$, the isotropy group $\Gamma_y$ is cyclic of order $r$.
	\end{enumerate}
	Let $\mf{X} =\mf{X}_{(\mc{O}_X(D),\, s_D,\, r)}$. There is an isomorphism of stacks $F: \mf{X} \xrightarrow{\simeq} [Y/\Gamma]$. 
\end{proposition}
\begin{proof}
	We briefly recall the isomorphism $F$ from \cite[Proposition 2.4]{BisMajWan12}. The additional assumptions ensure that we are in the situation of \cite[\S 2.2]{BisMajWan12}, so that the Proposition of [\textit{loc. cit.}] can be applied.  
	Suppose we are given an object of $\mf{X}$ over a scheme $T$, which means a morphism 
	\begin{align}
		\varphi: T \rightarrow \mf{X}
	\end{align}
	It is shown in \S 2.2 of [\textit{loc. cit.}] that there is a natural morphism $u: Y\rightarrow \mf{X}$, and that the projection morphism 
	$$T\times_{\mf{X}} Y \rightarrow T$$
	is a $\Gamma$-torsor \cite[Lemma 2.3]{BisMajWan12}. Then the top and left arrows of the Cartesian diagram
	\begin{align}
		\xymatrix{
			T\times_{\mf{X}} Y \ar[r]^(0.6){a} \ar[d]_{b} & Y \ar[d]^{u}\\
			T \ar[r]^{\varphi} & \mf{X}
		}
	\end{align}
	is the candidate for the corresponding object in $[Y/\Gamma](T)$. To go the other way, recall that the objects of $\mf{X}$ are tuples as described in \S\ref{subsection:root stack}. Suppose an object of $[Y/\Gamma]$ over a scheme $T$ is given by the data
	\begin{align}
		\xymatrix{
			P \ar[r]^{h} \ar[d]_{q} & Y \\
			T &
		}
	\end{align}
	where $q$ is a $\Gamma$-torsor map, and $h$ is $\Gamma$-equivariant. The $\Gamma$-invariant map $p\circ h : P\rightarrow X$ then factors through the torsor map $q$, yielding a morphism $f: T\rightarrow X$. This $f$ will be our first entry of the intended tuple corresponding to an object in $\mf{X}(T)$. The other entries of the tuple are also constructed similarly (see \cite{BisMajWan12} for more details).
\end{proof}

The next proposition shows that the isomorphism of stacks 
$F : \mf{X} \stackrel{\simeq}{\longrightarrow} [Y/\Gamma]$ 
in Proposition \ref{root stack isomorphic to quotient stack proposition} 
is compatible with their respective real structures discussed in 
Proposition \ref{real structure on root stack proposition} and 
\ref{real structure on quotient stack proposition}. 

\begin{proposition}\label{compatibility of real structures proposition}
	The isomorphism $F:\mf{X}\xrightarrow{\simeq} [Y/\Gamma]$ of Proposition \ref{root stack isomorphic to quotient stack proposition} is compatible with the real structures $\sigma_{\mf{X}}$ and $\sigma_{[Y/\Gamma]}$, i.e. the following diagram is $2$-commutative:
	\begin{align}
		\xymatrix@=1.3cm{
			\mf{X} \ar[r]^{\sigma_{\mf{X}}} \ar[d]_{F} & \mf{X} \ar[d]^{F} \\
			[Y/\Gamma] \ar[r]^{\sigma_{[Y/\Gamma]}} & [Y/\Gamma]
		}
	\end{align}
\end{proposition}
\begin{proof}
	
	Consider an object $\alpha$ of $\mf{X}(T)$ given by a morphism $\varphi: T\rightarrow \mf{X}$. Let us understand its image in $[Y/\Gamma](T)$ under $\sigma_{\mf{X}}(T)\circ F(T)$. As mentioned above, $\sigma_{[Y/\Gamma]}(T)(\alpha)$ consists of the top and the left arrows of 
	the Cartesian diagram
	\begin{align}
		\xymatrix@=1.5cm{
			T\underset{\varphi,\mf{X},u}{\times} Y \ar[r]^(0.6){h} \ar[d] & Y \ar[d]^{u} \\
			T \ar[r]^{\varphi} & \mf{X}
		}
	\end{align}
	where $h$ is the pull-back of $\varphi$ along $u$. Hence by Proposition \ref{real structure on quotient stack proposition}, $\sigma_{[Y/\Gamma]}(T)(F(T)(\alpha))$ is given by the data
	\begin{align}\label{data 1}
		\xymatrix@=1.5cm{
			\sigma_Y^*\left(T\underset{\varphi,\mf{X},u}{\times} Y\right) \ar[r]^(0.6){h'} \ar[d] & Y \\
			T & 
		}
	\end{align}
	where $h'$ is the pull-back of $h$ along $\sigma_Y$. But note that $h'$ is the pull-back of $\varphi$ along $u\circ\sigma_Y$, and since pull-back of $\varphi$ along the composition $u\circ\sigma_Y$ is isomorphic to taking two consecutive pull-backs, first taking pull-back of $\varphi$ along $u$ (namely $h$) followed by the pull-back of $h$ again by $\sigma_Y$, we get
	\begin{align}
		\sigma_Y^*\left(T\underset{\varphi,\mf{X},u}{\times} Y\right) & = \left(T\underset{\varphi,\mf{X},u}{\times} Y\right) \underset{h,Y,\sigma_{Y}}{\times} Y \nonumber \\ 
		\nonumber \hfill\\
		&\simeq T\underset{\,\,\,\varphi,\mf{X},u\circ \sigma_Y}{\times} Y \label{pullback iso1}
	\end{align}
	and hence $\sigma_{[Y/\Gamma]}(T)(F(T)(\alpha))$ can also be expressed by the data
	\begin{align}\label{data 2}
		\xymatrix@=1.5cm{
			T\underset{\,\,\,\varphi,\mf{X},u\circ \sigma_Y}{\times} Y \ar[r]^(0.6){h'} \ar[d] & Y \\
			T & 
		}
	\end{align}
	where $h'$ is the pull-back of $\varphi$ along $u\circ\sigma_Y$.\\
	On the other hand, $\sigma_{\mf{X}}(T)(\alpha)$ corresponds to the morphism
	$$\sigma_{\mf{X}}\circ \varphi : T\rightarrow \mf{X}\,,$$
	and thus $F(T)\left(\sigma_{\mf{X}}(T)(\alpha)\right)$ is given by the top and left arrows of the diagram
	\begin{align}\label{data 3}
		\xymatrix{
			T \underset{\sigma_{\mf{X}}\circ \varphi,\mf{X},u}{\times} Y \ar[rr] \ar[d] && Y \ar[d]^{u}\\
			T \ar[r]^{\varphi} & \mf{X} \ar[r]^{\sigma_{\mf{X}}} & \mf{X} 
		}
	\end{align}
	But again, since the pull-back of $u$ along $\sigma_{\mf{X}}\circ\varphi$ is isomorphic to taking two consecutive pullbacks, first taking the pull-back of $u$ along $\sigma_{\mf{X}}$ followed by taking the pull-back of the resulting fiber product (say $u':\mf{X}\underset{\sigma_{\mf{X}},Y,\sigma_Y}{\times}Y\rightarrow \mf{X}$) along $\varphi$, we can write 
	\begin{align*}
		T \underset{\sigma_{\mf{X}}\circ \varphi,\mf{X},u}{\times} Y  & \simeq T \underset{\varphi,\mf{X},u'}{\times}\left(\mf{X}\underset{\sigma_{\mf{X}},Y,\sigma_Y}{\times}Y\right)\,.
	\end{align*}
	But note that $$\mf{X}\underset{\sigma_{\mf{X}},Y,\sigma_Y}{\times}Y \simeq Y\,,$$
	since it is easy to see that the diagram
	\begin{align*}
		\xymatrix{
			Y \ar[r]^{\sigma_Y} \ar[d]_{p} & Y \ar[d]^{p} \\
			X \ar[r]^{\sigma} & X
		}
	\end{align*}
	is Cartesian, and consequently, by Lemma \ref{pullback of atlas}, the following diagram is Cartesian as well:
	\begin{align*}
		\xymatrix{
			Y \ar[r]^{\sigma_Y} \ar[d]_{u} & Y \ar[d]^{u} \\
			\mf{X} \ar[r]^{\sigma_{\mf{X}}} & \mf{X}
		}
	\end{align*}	
	It then follows that 
	\begin{align*}
		T \underset{\sigma_{\mf{X}}\circ \varphi,\mf{X},u}{\times} Y  & \simeq T \underset{\varphi,\mf{X},u'}{\times}\left(\mf{X}\underset{\sigma_{\mf{X}},Y,\sigma_Y}{\times}Y\right)\,\,\,\,\left[\textnormal{Here}\,u'\,\,\textnormal{is the projection}\,\,\mf{X}\underset{\sigma_{\mf{X}},Y,\sigma_Y}{\times}Y\rightarrow \mf{X}\right]\\
		& \simeq T\underset{\varphi,\mf{X},u}{\times} Y
	\end{align*}
	
	Thus $F(T)\left(\sigma_{\mf{X}}(T)(\alpha)\right)$ can be re-written as
	\begin{align}\label{data 4}
		\xymatrix@=1.5cm{
			T\underset{\varphi,\mf{X},u}{\times} Y \ar[r]^(0.6){\sigma_Y\circ \varphi'} \ar[d] & Y\\
			T &
		}
	\end{align}
	where $\varphi'$ is the pull-back of $\varphi$ along $u:Y\rightarrow \mf{X}$.	
	Finally, comparing \eqref{data 2} with \eqref{data 4} and using the Cartesian diagram for \eqref{pullback iso1}, namely
	\begin{align*}
		\xymatrix@=1.5cm{
			T\underset{\,\,\,\varphi,\mf{X},u\circ \sigma_Y}{\times} Y \ar[r]^(0.6){h'} \ar[d]_{l} & Y \ar[d]^{\sigma_Y} \\
			T\underset{\varphi,\mf{X},u}{\times} Y \ar[r]^{h} & Y
		}
	\end{align*}
	we see that the triangles in the diagram below commute:
	
	\begin{align*}
		\xymatrix@=2.5cm{
			T\underset{\,\,\,\varphi,\mf{X},u\circ \sigma_Y}{\times} Y \ar[rrd]^{h'} \ar[rd]_{l} \ar[dd] & & \\
			& T\underset{\varphi,\mf{X},u}{\times} Y \ar[ld] \ar[r]^{\sigma_Y\circ \varphi'} & Y \\
			T & &
		}
	\end{align*}
	Thus $l$ maps $\sigma_{[Y/\Gamma]}(T)(F(T)(\alpha))$ to $F(T)\left(\sigma_{\mf{X}}(T)(\alpha)\right)$. It is clear that $l$ is an isomorphism from the fact that it is the pull-back of the isomorphism $\sigma_Y$.
	This yields a natural isomorphism of functors between $F\circ \sigma_{\mf{X}}$ and $\sigma_{[Y/\Gamma]}\circ F$, which proves our claim.
\end{proof}

\section*{Acknowledgment}
This work was initiated during the first named author's visit in the 
Department of Mathematics and Statistics of the Indian Institute of Science Education 
and Research Kolkata in December 2022. 
The second named author is partially supported by the DST INSPIRE Faculty Fellowship 
(Research Grant No.: DST/INSPIRE/04/2020/000649), Ministry of Science \& Technology, 
Government of India.

\end{document}